\documentclass[12pt,a4paper]{article}
\usepackage[pagewise]{lineno}
\usepackage[table]{xcolor}
\usepackage{diagbox}
\usepackage{amsmath}
\usepackage{cancel}
\usepackage{graphicx}
\usepackage{subcaption}
\usepackage{amsmath,amssymb,amsthm,graphics,amsfonts}
\usepackage{multirow}
\usepackage{csquotes}

\nonstopmode\numberwithin{equation}{section}

\newtheorem{Lemma}{Lemma}[section]
\newtheorem{Theorem}[Lemma]{Theorem}
\newtheorem{theorem}{Theorem}

\newtheorem{Proposition}[Lemma]{Proposition}

\renewcommand{\qed}{\rule{6pt}{6pt}}
\newenvironment{Proof}{\noindent{\bf Proof:}}{\qed\bigskip}
 
 \usepackage{lipsum}
\usepackage{color}

 \newcommand{\bea}{\begin{eqnarray}}
 \newcommand{\eea}{\end{eqnarray}}

\usepackage{graphicx,latexsym}
\usepackage{lscape}

\begin{document}
\title{\LARGE Existence of isolated periodic waves of a family of PDEs with cubic reaction term}
\author{ \Large{
Krishna Patra, 
Ch. Srinivasa Rao}\\
Department of Mathematics, Indian Institute of Technology Madras\\
Chennai 600036, India}
\date{}
\maketitle
\begin{center}{\bf{Abstract}}
\end{center}
In this paper we study the isolated periodic traveling wave solutions of a family of reaction-convection-diffusion equations with cubic reaction term. Existence/nonexistence of periodic traveling wave solutions are discussed in different parametric ranges. The monotonicity of the ratio of Abelian integrals is used to prove the existence of at most one limit cycle. Finally, numerical study is presented at the end.
\vskip 0.1in
\noindent \textit{Keywords:}
Isolated periodic wave, Hamiltonian system, Limit cycle, Abelian integral.
\vskip 0.1in  
\noindent \textit{Mathematics Subject Classiﬁcation (2020):}
34C05, 34C07, 34C08, 35C07, 37G15
\section{Introduction}
In recent times, reaction-convection-diffusion equations have become prevalent across various scientific and engineering disciplines especially within fluid dynamics, heat transfer and population ecology.
Petrovskii and Li \cite{{ref2}} considered the following model
\begin{equation} \label{kp3}
 \frac{\partial u(x,t)}{\partial t}+(a_{0}+2a_{1}u)\frac{\partial u}{\partial x}=\tilde {D}\frac{\partial^2 u}{\partial x^2} + u(u-u_0)(k-u), 
\end{equation}
where $a_0,a_1,u_0,k,\tilde{D}$ are real numbers. This model (\ref{kp3}) has significant application in Ecology ( see more details in \cite{ref1}). Later Sun et al. \cite{ref3} proved that the model (\ref{kp3}) has at most one periodic traveling wave solution. 
Inspired by the work of Petrovskii and Li \cite{{ref2}} and Sun et al. \cite{ref3}, we consider a generalization of (\ref{kp3}):
\begin{equation} \label{eq3.01}
\frac{\partial u}{\partial t}+(a_{0}+a_{n}u^n)\frac{\partial u}{\partial x}=D\frac{\partial^2 u}{\partial x^2}+\tilde \beta u(u-u_{0})(k-u),
\end{equation}
where $\tilde \beta,$ $ D $ are positive real numbers, $n$ is a positive integer greater than $1$, $a_0$, $a_n$, $u_0$ and $k$ are real numbers. This generalized model may be useful for ecologists.

Wang et al. \cite{ref4} considered the model 
\begin{equation} \label{eq2.01}
\frac{\partial u}{\partial t}+(a_{0}+2a_{1}u+3a_{2}u^2)\frac{\partial u}{\partial x}=D\frac{\partial^2 u}{\partial x^2}+\beta u(u-u_{0})(k-u),
\end{equation}
where $\beta, D $ are positive real numbers, and  $a_{0},$  $a_{1}$ and $a_2$ are real numbers. They proved that the equation (\ref{eq2.01}) has two isolated periodic traveling wave solutions simultaneously.

Recently Patra and Rao \cite{ref8} studied the model
\begin{equation} \label{eqn1.9}
\frac{\partial u}{\partial t}=D\frac{\partial^2 u}{\partial x^2}-(a_{0}+a_{1}(p+1)u^p)\frac{\partial u}{\partial x}+\tilde\beta u(u^2-1)(u^2+\gamma),
\end{equation}
 where $\gamma ,$ $D$ are positive real numbers, $p \in \mathbb{Z}^+$, $a_0$ and $a_1$ are real numbers and $\tilde\beta$ is a nonzero real number. They determined conditions on parameters such that the equation has at most one isolated periodic traveling wave solution. 
 
 Wang et. al.\cite{ref17} studied the partial differential equation
\begin{equation}\label{eqn1.10}
  \frac{\partial u}{\partial t}+\left(\alpha_0+\sum_{i=1}^n\alpha_iu^i\right)\frac{\partial u}{\partial x}=\frac{\partial^2 u}{\partial x^2}+\beta u(1- u^n)(u^n-\gamma),   
\end{equation}
where $\alpha_i,\beta,\gamma$ are real numbers and $n$ is a positive integer. They have proved that the equation (\ref{eqn1.10}) has multiple isolated periodic waves for certain parameters.

Liu and Han \cite{{ref9}} studied existence of periodic travelling wave solutions of a generalized Burgers-Fisher equation
\begin{equation}{\label{kp100}}
 \frac{\partial u}{\partial t}+\alpha u^m\frac{\partial u}{\partial x}+\beta \frac{\partial^2 u}{\partial x^2}+\gamma u(1-u^m)=0,    
\end{equation}
where $m$ is a positive integer, and $\alpha,\beta$, and $\gamma$ are real numbers. They proved that :
(i) for odd positive integer $m,$ there exists either an isolated periodic wave solution or a solitary wave solution  under some constraints.\\
(ii) for even positive integer $m,$ there exists either an isolated periodic wave solution or a kink or antikink wave solution under some constraints.\\
Patra and Rao \cite{ref5} studied the partial differential equation
\begin{equation}
\frac{\partial u}{\partial t}+\alpha u^m \frac{\partial u}{\partial x}+\beta \frac{\partial^2 u}{\partial x^2}+\delta u^l(1-u^n)=0.
\label{kpr01}
\end{equation}
where $ \alpha,\delta \in \mathbb{R},$ $\beta \neq 0,$ and $m,$ $l$ and $n$ are positive integers. They examined that the periodic wave solution of equation (\ref{kpr01}) for different cases. For $m=1,$ Zhang and Xia \cite{ref11} proved existence of periodic wave solution of equation (\ref{kpr01}) for some values of $l$ and $n.$ \\ 
Wei  and Chen \cite{ref16} considered the reaction–convection–diffusion equation
\begin{equation} \label{eq1.11}
\frac{\partial u}{\partial t}=\frac{\partial^2 u}{\partial x^2}+(a_{0}+a_{1}u)\frac{\partial u}{\partial x}+\delta_{1} u(u^q+\delta_{2} )(u^q+\beta),
\end{equation}
where $\delta_{1,2}= \pm 1,$ $\beta \neq 0,$  $0< | a_{1}|\ll 1,$ and $q \in \mathbb{Z}^+.$ When $q=3,4$, a comprehensive analysis was conducted to determine the conditions for the existence and number of near-ordinary periodic wave solutions, addressing both monotone and non-monotone cases of the Abelian integral ratio. Later, Wei et. al. \cite{ref10} studied the reaction–convection–diffusion equation
\begin{equation} \label{eq1.12}
\frac{\partial u}{\partial t}=\frac{\partial^2 u}{\partial x^2}+(a_{0}+a_{1}u^n)\frac{\partial u}{\partial x}+\delta_{1} u(u^q-1 )(u^q+\beta),
\end{equation}
where $\beta \neq 0,$  $0< | a_{1}|\ll 1,$ and $q,n \in \mathbb{Z}^+.$ They demonstrated that \rm(i) for any $q\in \mathbb{Z}^+$, there is at most one reachable near-ordinary periodic wave solution of (\ref{eq1.12}) when $\delta_{1}=1,$ $n=1$, $\beta> 0 ,$ \rm(ii) for any $n>1,$ there is at most one reachable near-ordinary periodic wave solution of (\ref{eq1.12}) when $\delta_{1} = -1$, $q=1$, $\beta=-2.$

In this paper we obtain the conditions on the parameters $k$, $u_0$, $a_0$, and $a_n$ such  that the model has an isolated periodic wave solution. As Sun et. al. \cite{ref3} proved the existence of periodic traveling wave solution of (\ref{eq3.01}) for $n=1,$ we study for integer  $n>1.$ We  mainly focus on four cases \rm(i) $0<u_0<k \leq 2u_0,$ \rm(ii)$-k<u_0<0,$ \rm(iii) $u_0<-k,$ and \rm(iv) $u_0=k>0$ and $u_0=-k<0.$
Let 
\begin{equation}\label{kpr2.2}
u(x,t) =u(\eta),\hspace{1cm} \eta = x-ct.
\end{equation}
Here $c$ is the wave speed. 
Now substituting (\ref{kpr2.2}) into (\ref{eq3.01}), we get
\begin{equation}\label{kpr2.3}
 D\frac{d^2u}{d\eta^2}=(-c+a_0+a_n u^n)\frac{du}{d\eta}-\tilde\beta u(u-u_0)(k-u).
\end{equation}
Rewriting the equation (\ref{kpr2.3}) as a system of first order ordinary differential equations, we have 
\begin{equation}\label{kpr2.4}
 \frac{du}{d\eta}=y,\hspace{.5cm}
 \frac{dy}{d\eta}=\frac{1}{D}(-c+a_0+a_n u^n)y-\frac{\tilde\beta}{D} u(u-u_0)(k-u).
\end{equation}
Assume that $\frac{\tilde\beta}{D}=\beta,$ and $\frac{a_0-c}{D}= \alpha_0\epsilon,$ and $ \frac{a_n}{D} = \alpha_n\epsilon,$
where $0 <\epsilon<<1~\text{and}~\alpha_0$ and $ \alpha_n$ are (real) parameters. Then  we have the system of differential equations:
\begin{equation}\label{kpr2.5}
\frac{du}{d\eta}=y,\hspace{.5cm}
 \frac{dy}{d\eta}=\epsilon(\alpha_0+\alpha_nu^n)y-\beta u(u-u_0)(k-u).
\end{equation}
Now the corresponding unperturbed system [Equation (\ref{kpr2.5}) with  $\epsilon= 0$] is given by 
\begin{equation}\label{kpr1.17}
\frac{du}{d\eta}=y,\hspace{.5cm}
 \frac{dy}{d\eta}=-\beta u(u-u_0)(k-u).
\end{equation}
The system (\ref{kpr1.17}) is clearly a Hamiltonian system \cite{ref20} with Hamiltonian function $$H(u,y):=\frac{y^2}{2}+\beta\left(-\frac{u^4}{4}+\frac{(k+u_0)}{3}u^3-\frac{ku_0}{2} u^2\right):= \Psi(y)+\Phi(u).$$ 
The following two Theorems are used to classify the nature of the fixed points of the system (\ref{kpr1.17}).
\begin{theorem}(Andronov et al.\cite{ref02}
, Perko\cite{ref15}){\label{kp.r1}}
Let us consider the system of differential equations
\begin{equation}
 \frac{du}{d\eta}=y,\hspace{.5cm}
 \frac{dy}{d\eta}=\hat{\lambda} (u-u_0)^k(1+h(u)),
  \label{a01} 
\end{equation}

where $h(u)$ is an analytic function in a neighborhood of $u=u_0$, $h(u_0)=0,$ $k \geq 2,$ and $\hat{\lambda} \neq 0.$ Then 
\begin{itemize}
    \item [\rm(i)] $(u_0, 0)$ is a cusp if $k$ is even;
    \item [\rm(ii)] $(u_0, 0)$ is a saddle if $k$ is odd and $\hat \lambda>0$;
    \item [\rm(iii)] $(u_0, 0)$ is a center or focus if $k$ is odd and $ \hat \lambda<0$;
\end{itemize}
\end{theorem}
\begin{theorem}(Hale \cite{ref24}){\label{kp.r2}}
Suppose $\textbf{f}:\mathbb{R}^2\to \mathbb{R}^2$ is a continuous function and for any $\textbf{x}_0 \in \mathbb{R}^2$ the system 
 \begin{equation} 
 \dot{\textbf{x}}=\textbf{f}(\textbf{x}),
\label{a1}
\end{equation} 
has a unique solution $\textbf{x}=\textbf{x}(t,\textbf{x}_0)$ with $\textbf{x}(0,\textbf{x}_0)=\textbf{x}_0.$ Assume that the system (\ref{a1}) is conservative and $\textbf{a}\in \mathbb{R}^2$ is an equilibrium point of the system (\ref{a1}).\\
If $E(\textbf{x})$ is an integral of (\ref{a1}) in a bounded open neighborhood $U\subset \mathbb{R}^2$ of the equilibrium point  $\textbf{a}$ such that $E(\textbf{x})> E(\textbf{a})$ for  $\textbf{x} \in U\setminus{\{\textbf{a}\}},$ then the equilibrium point  $\textbf{a}$ is a center.
\end{theorem}
In the first instance, we find all the centers of the unperturbed system (\ref{kpr1.17}). For each center of the unperturbed system (\ref{kpr1.17}), we identify the periodic annulus. Note that a limit cycle of the system (\ref{kpr2.5}) corresponds to an isolated periodic traveling wave of the equation (\ref{eq3.01}). Therefore it is enough to study the existence of limit cycle of the perturbed Hamiltonian system.

Let us consider a perturbed Hamiltonian system $X_{H,\epsilon}$:
\begin{align}
 \dot {x}= H_{y}(x,y)+\epsilon f(x,y),\hspace{.5cm}
 \dot {y}=-H_{x}(x,y)+\epsilon g(x,y)
  \label{kp1req2.7}  
\end{align}
where $\epsilon>0$ is sufficiently small, $f(x,y)$ and $g(x,y)$ are polynomial functions in $x$ and $y,$ and  $\epsilon=0$ gives the corresponding Hamiltonian system.
Abelian integral is an important tool to determine the number of limit cycles of the perturbed system. Now the Abelian integral corresponding to the perturbed system (\ref{kp1req2.7}) is
\begin{align}
   & A(h)=\oint_{\Gamma_ h}\,g(x,y)\, dx-f(x,y)\, dy.
   \label{eqn1.17}
\end{align}
See Zoladek and Murilo \cite{ref21}, Patra and Rao \cite{ref8} for more detailed explanations. Here $\Gamma_h$ is a closed trajectory of the unperturbed system (\ref{kpr1.17}) and  $\Gamma_h: H(u,y)=h$ for some $h \in (p_1,p_2)$,  where $(p_1,p_2)$ is the maximal interval so that the family $\{\Gamma_h:H(u,y)=h,h \in (p_1,p_2)\}$ of trajectories forms a periodic annulus. The Poincar\'{e}-Pontryagin theorem (see Poincar\'{e} \cite{ref31} and Pontryagin \cite{ref32}) and Theorem 2.4 of \cite{ref30} indicate that if $A(h)$ is not identically zero, the number of zeros of $A(h)$ provides an upper bound on the number of limit cycles of the system (\ref{kp1req2.7}) that bifurcate from the periodic annulus.

The Abelian integral corresponding to the perturbed system (\ref{kpr2.5}) is 
\begin{eqnarray}
 A(h) &=& \oint_{\Gamma_ h}\,(\alpha_0+\alpha_nu^n)y\, du\nonumber\\
 &=&\alpha_0 \left(\oint_{\Gamma_ h}\,y\, du\right) +\alpha_n \left(\oint_{\Gamma_ h}\,u^n y\, du\right)\nonumber\\
&:=& \alpha_0 A_0(h)+\alpha_n A_n(h).
\end{eqnarray}
\textbf{Note:} Now $ |A_0(h)|=|\oint_{\Gamma_h}\,y\, du|=$ Area of the region enclosed by the closed curve $\Gamma_h$ due to Green's Theorem in the plane. So $A_0(h) \neq 0$ and hence the  ratio $G_{n}(h):=\frac{A_n(h)}{A_0(h)}$ is well defined.

The next two Theorems are used to prove the existence of isolated periodic waves.

\begin{theorem}(Christopher and Li \cite{ref30})\label{lemma3.1}
Assume that $A(h) \not\equiv0$ for $h\in(p_1,p_2)$. Then the following statements hold.\\
\rm{(i)} If a limit cycle of the system $X_{H,\epsilon}$ bifurcates from $\Gamma_{\hat h}$, then $A(\hat h)=0.$\\
\rm{(ii)} If there exists  $\tilde h\in(p_1,p_2)$ such that $A(\tilde h)=0$ and $A'(\tilde h)\neq0$. Then the system $X_{H,\epsilon}$ has a unique limit cycle that bifurcates from $\Gamma_{\tilde h}.$
\end{theorem}
\begin{theorem}{(Liu et al. \cite{ref6})}{\label {thmc}}
Consider the system of differential equations
\begin{equation}{\label {eqn1.13}}
\frac{du}{d\eta}=y,\hspace{.5cm}
 \frac{dy}{d\eta}=\epsilon(\alpha_0 +\alpha_n \tilde g_n(u))y- \Phi'(u), 
\end{equation}
where $\epsilon>0$ is sufficiently small, $\tilde g_n(u)$ is a polynomial, and $\alpha_0,\,\alpha_m$ are real numbers. \\
Suppose that: 
\begin{itemize}
\item[(a)] The Hamiltonian function $H(u,y)$ corresponding to the unperturbed system of $(\ref{eqn1.13})$ has the form $H(u,y) := \Psi(y)+\Phi(u)$ and satisfies $\Phi'(u)(u-\hat a)>0\,\, (\,\, { or} < 0\,\,)$ { for} $u\in(\alpha,B)\setminus \{\hat a\}$ where 
$(\hat a,0)$ is a center of the unperturbed system of $(\ref{eqn1.13})$ and $B$ can be determined by the relation $\Phi(\alpha)=\Phi(B).$\\
(Note that, in view of the condition $(a)$, there exists an involution $\delta$ defined on $(\alpha,B)).$
\item[(b)]
$\hat{A}_{n}(h):=\frac{A_n(h)}{A_0(h)}=\frac{\oint_{\Gamma_h}\,\tilde g_n(u)y\, du}{\oint_{\Gamma_h}\,y\, du}$ and $T_n(u):=(n+1)\frac{\int_{\delta(u)}^{u} \tilde g_n(t) \,dt }{\int_{\delta(u)}^{u} \,dt }$, where $\Gamma_h$ is a closed trajectory of the unperturbed system of $(\ref{eqn1.13})$ and  $\Gamma_h: H(u,y)=h$ 
 for some $h \in (p_1,p_2)$, where $(p_1,p_2)$ is the maximal interval so that the family $\{\Gamma_h:H(u,y)=h,h \in (p_1,p_2)\}$ forms a periodic annulus, and $\delta$ is an involution  defined on $(\alpha,B)).$
\end{itemize}
Then the strict monotonicity of ${T_n}(u)$ on $(\hat a,B)$ implies that $\hat{A}(h)$ is strictly monotonic on $(p_1, p_2)$.
\end{theorem} 

This paper is structured as follows: 
Section 1 presented the introduction. 
Section 2 analyzes the nature of all the equilibrium points of the unperturbed system (\ref{kpr1.17}) and identifies the periodic annuli. 
Section 3 discusses the monotonicity of the function $G_n(h)$ and then presents the existence of limit cycles of (\ref{kpr2.5}) for relevant parametric ranges. 
Section 4 examines the necessary conditions required on the ratio $\frac{\alpha_0}{\alpha_n}$  for the existence of limit cycles. Finally, Section 5 puts forward the conclusions.

\section{Classification of equilibrium points and phase portraits of the unperturbed system (\ref{kpr1.17})}
It is easy to see that $(0,0)$, $(u_0,0)$ and $(k,0)$ are the fixed points of the unperturbed system (\ref{kpr1.17}). Now we will classify all the fixed points for different cases depending on the parameters of the system. The Jacobian matrix evaluated at 
the fixed point $(0,0)$ is 
$\begin{pmatrix}
0 & 1\\
\beta k u_0 & 0
\end{pmatrix}$ and its determinant is $-\beta k u_0.$ The determinants of the Jacobian matrix evaluated at the points $(u_0,0)$ and $(k,0)$ are $\beta u_0 (k-u_0)$ and  $\beta k (u_0-k)$, respectively.
\subsection{ Case: $0<u_0<k$}
In view of the  linear analysis, the fixed points $(0,0)$ and $(k,0)$ are saddle points and further investigation is required to determine whether $(u_0,0)$ is a center for the system (\ref{kpr1.17}). Clearly the unperturbed system (\ref{kpr1.17}) is a Hamiltonian system.
The trajectories of the system are given by $$H(u,y):=\frac{y^2}{2}+\beta\left(-\frac{u^4}{4}+\frac{(k+u_0)}{3}u^3-\frac{ku_0}{2} u^2\right)=h.$$ 

Let us investigate whether $(u_0,0)$ is a center for the system (\ref{kpr1.17}). Note that
\begin{eqnarray}
&&H_{u}(u,y)=\beta u(u-u_0)(k-u),\,\,H_y(u,y)=y, \nonumber\\
&&H_{u y}(u,y) =H_{y u}(u,y) =0, \,\,H_{u u}(u,y)=\beta (-3u^2+2(u_0+k)u-ku_0)\,\, ,~ H_{ y y}(u,y)=1, \nonumber\\
&&\tilde{H}(u,y):= H_{u u}(u,y)H_{ y y}(u,y)-(H_{u y}(u,y))^2.\nonumber 
\end{eqnarray}
Observe that $\tilde H(u_0,0)=\beta u_0(k-u_0)>0$ and $H_{uu}(u_0,0)=\beta u_0(k-u_0)>0.$ This implies that $(u_0,0)$ is an isolated minimum and hence $(u_0,0)$ is a center due to Theorem \ref{kp.r2}.\\

Now we are interested to find the interval $(h_1,h_2)$ such that the trajectories 
$$\{\Gamma_{h}:H(u,y)=h, \,\,h\in (h_1,h_2)\}$$ form a periodic annulus. 
Let us consider all the closed trajectories around $(u_0,0)$. The line $u=u_0$ intersects the trajectories $H(u,y)=h$ with $h> \frac{\beta u_{0}^3}{12}(u_0-2k) $ at two points in the $u-y$ plane for $y=\pm \sqrt{2} \sqrt{h-\frac{\beta u_{0}^3}{12}(u_0-2k)}$. Note that $H(u,y)=\frac{\beta u_{0}^3}{12}(u_0-2k)$ corresponds to the center $(u_0,0).$ Recall that $(k,0)$ and $(0,0)$ are saddles and $H(k,0)=\frac{\beta k^3}{12}(k-2u_0)\leq 0$ when $k \leq 2u_0$, and $H(0,0)=0$. The trajectory $H(u,y)=\frac{\beta k^3}{12}(k-2u_0)$ corresponds to the homoclinic orbit connecting $(k,0)$ when $k<2u_0.$ Also when $k=2u_0,$ $H(u,y)=0$ corresponds to the heteroclinic orbits connecting $(0,0)$ and $(k,0).$ This implies that the interval for the periodic annulus around the fixed point $(u_0,0)$ is given by $(h_1,h_2)=\left(\frac{\beta u_{0}^3}{12}(u_0-2k),\frac{\beta k^3}{12}(k-2u_0)\right)$ when $k \leq 2u_0.$ 

Similarly, for $k>2u_0,$  one can observe that the periodic annulus around $(u_0,0)$ is formed by the trajectories $H(u,y)=h$, $h \in (h_1,h_2)=\left(\frac{\beta u_{0}^3}{12}(u_0-2k),0\right).$

Phase portraits for the cases $k<2u_0$, $k=2u_0$ and $k>2u_0$ are presented in 
Figures  \ref{kpr001},  \ref{kpr02}, and \ref{kpr001a}, respectively. 

\begin{figure}
  \centering
      \includegraphics[width=1.08\linewidth]{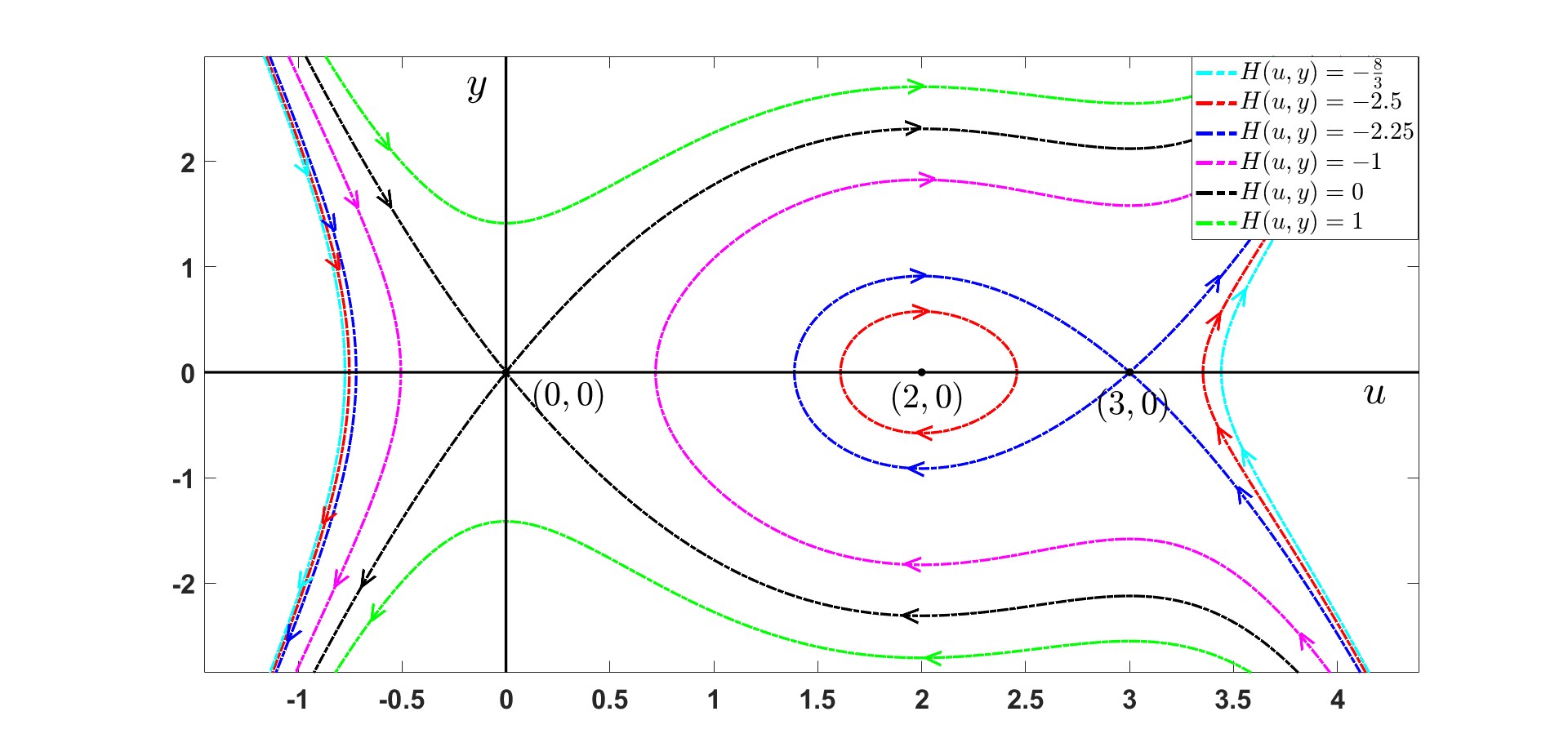}
  \caption{Phase portrait of the system (\ref{kpr1.17}) when $\beta=1,$ $u_0=2,$ and $k=3.$}
  \label{kpr001}
\end{figure}
\begin{figure}
  \centering
      \includegraphics[width=1.08\linewidth]{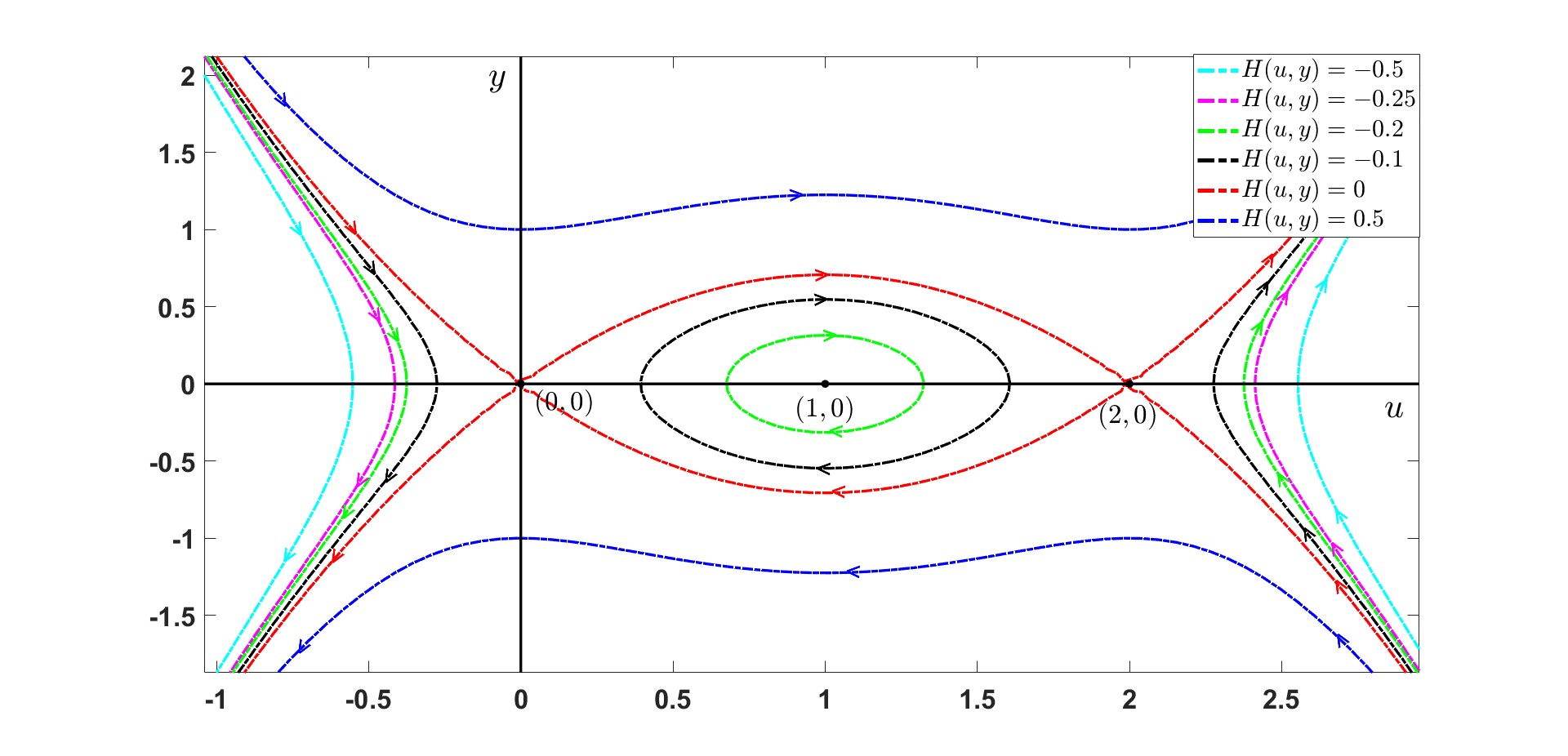}
  \caption{Phase portrait of the system (\ref{kpr1.17}) when $\beta=1,$ $u_0=1,$ and $k=2.$}
  \label{kpr02}
\end{figure}
\begin{figure}
  \centering
      \includegraphics[width=1.08\linewidth]{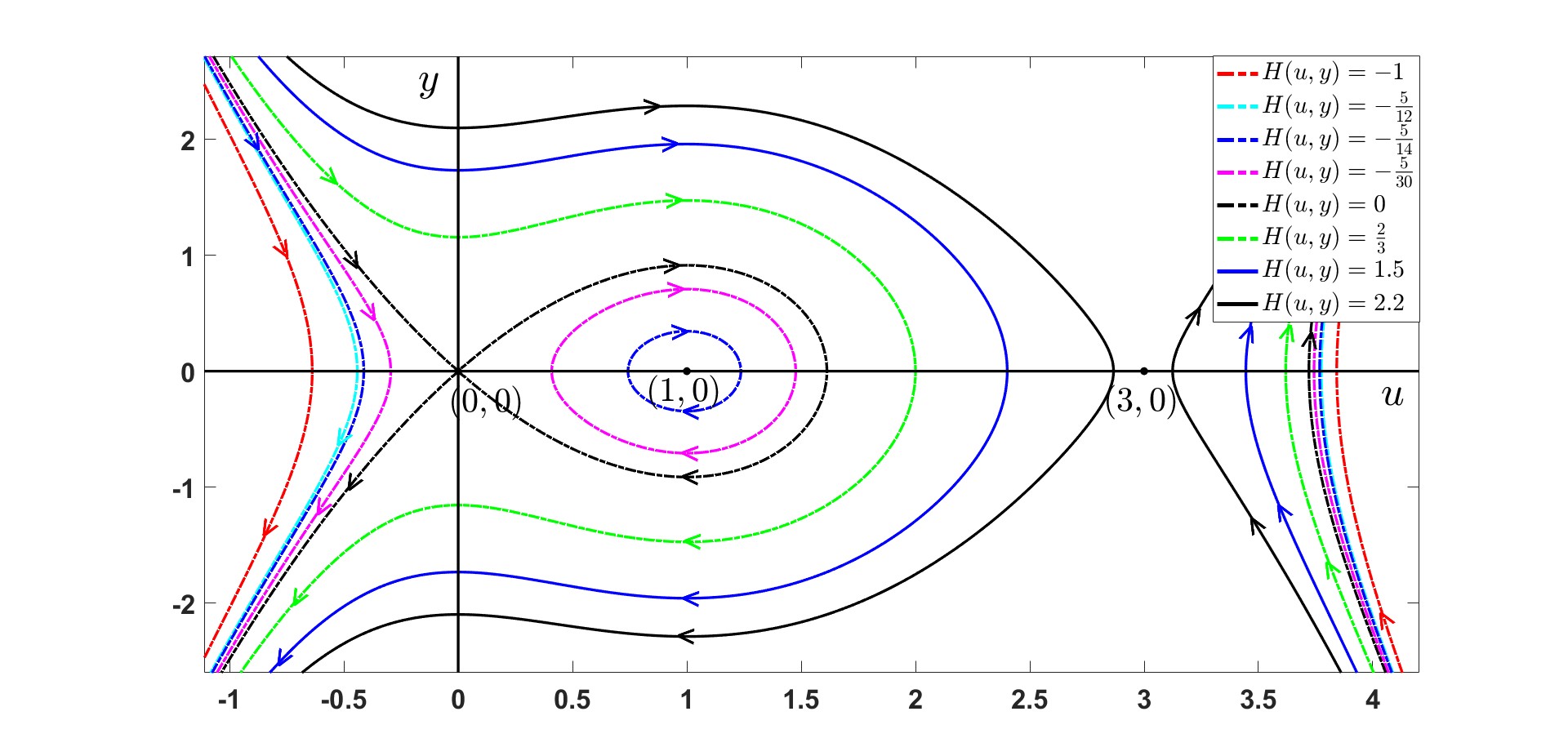}
  \caption{Phase portrait of the system (\ref{kpr1.17}) when $\beta=1,$ $u_0=1,$ and $k=3.$}
  \label{kpr001a}
\end{figure}
\subsection{Case: $-k<u_0<0$}
In this case  the fixed points $(u_0,0)$ and $(k,0)$ are saddle points due to the linear analysis and further investigation is required to determine whether $(0,0)$ is a center for the system (\ref{kpr1.17}). Note that $H_{uu}(0,0)=-\beta ku_0>0$ and $\tilde H(0,0)=-\beta k u_0>0.$ This implies that $(0,0)$ is an isolated minimum and hence  the fixed point $(0,0)$ is a center in view of  Theorem \ref{kp.r2}.

Let us consider all the closed trajectories around $(0,0)$. The line $u=0$ intersects the trajectories $H(u,y)=h>0$ at two points in the $u-y$ plane for $y=\pm \sqrt{2h}.$ Note that  $H(u,y)=0$ corresponds to the center $(0,0).$ Recall that $(k,0)$ and $(u_0,0)$ are saddles. 
Further $H(u_0,0)=\frac{\beta u_{0}^3}{12}(u_0-2k)> 0$ and $ H(k,0)=\frac{\beta k^3}{12}(k-2u_0)>H(u_0,0)$. Then $H(u,y)=\frac{\beta u_{0}^3}{12}(u_0-2k)$ corresponds to the homoclinic orbit connecting $(u_0,0)$. This, in turn, implies that the trajectories $\Gamma_h$, $h \in (h_1,h_2)=\left(0,\frac{\beta u_{0}^3}{12}(u_0-2k)\right)$ form the periodic annulus around the fixed point $(0,0).$ A typical phase portrait for this  case is presented in Figure \ref{kpr002}.
\begin{figure}
  \centering
      \includegraphics[width=1.08\linewidth]{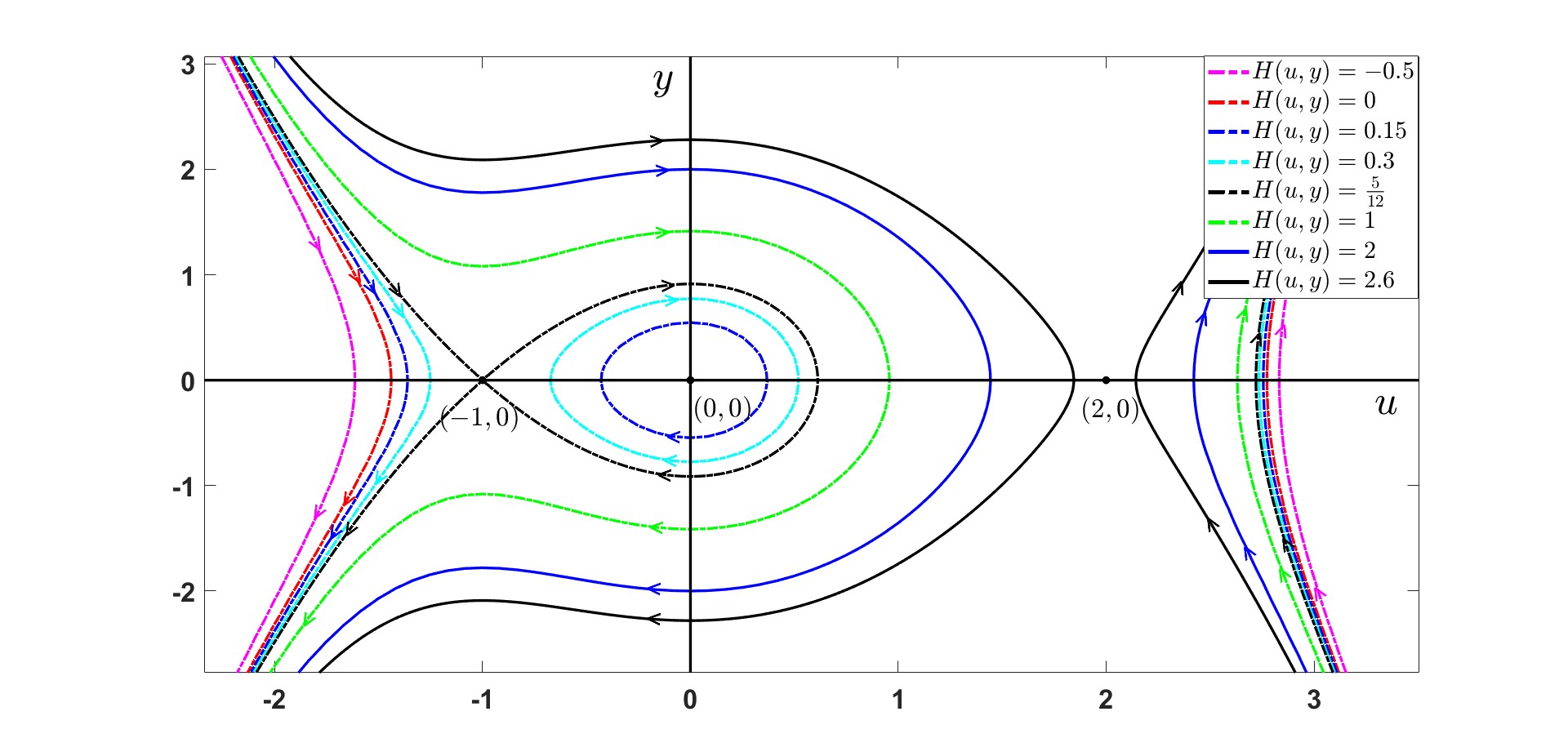}
  \caption{Phase portrait of the system (\ref{kpr1.17}) when $\beta=1,$ $u_0=-1,$ and $k=2.$}
  \label{kpr002}
\end{figure}

\subsection{Case: $u_0<-k$}
Fixed points $(u_0,0)$ and $(k,0)$ are saddle points in view of the linear analysis and further investigation is required to determine whether $(0,0)$ is center for the system (\ref{kpr1.17}).  In this case $H_{uu}(0,0)=-\beta ku_0>0$ and $\tilde H(0,0)=-\beta k u_0>0.$ This implies that $(0,0)$ is an isolated minimum and hence $(0,0)$ is a center due to Theorem \ref{kp.r2}.

Let us consider all the closed trajectories around $(0,0)$. The line $u=0$ intersects the trajectories $H(u,y)=h>0$ at two points in the $u-y$ plane for $y=\pm \sqrt{2h}.$ Further  $H(u,y)=0$ corresponds to the center $(0,0).$ Recall that the fixed points $(k,0)$ and $(u_0,0)$ are saddles. Note that  $ H(k,0)=\frac{\beta k^3}{12}(k-2u_0)>0$ and $H(u_0,0)=\frac{\beta u_{0}^3}{12}(u_0-2k)> H(k,0).$ Then $H(u,y)=\frac{\beta k^3}{12}(k-2u_0)$ corresponds to the homoclinic orbit connecting $(k,0)$. This implies that the trajectories $\Gamma_h,$ $h\in(h_1,h_2)=\left(0,\frac{\beta k^3}{12}(k-2u_0)\right)$ form the periodic annulus around the fixed point $(0,0).$ A representative phase portrait for this case is given in Figure \ref{kpr03}.
\begin{figure}
  \centering
      \includegraphics[width=1.08\linewidth]{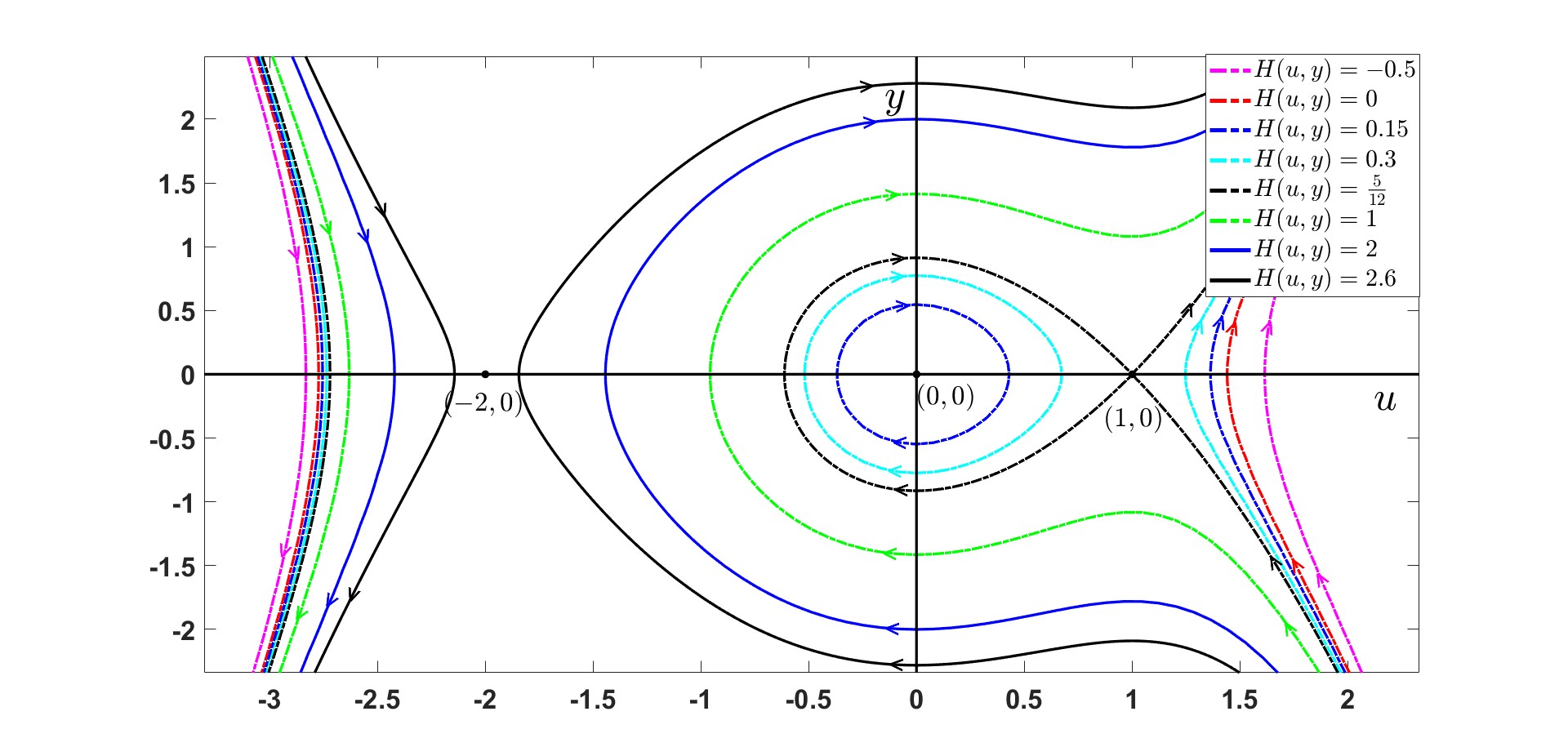}
  \caption{Phase portrait of the system (\ref{kpr1.17}) when $\beta=1,$ $u_0=-2,$ and $k=1.$}
  \label{kpr03}
\end{figure}

\subsection{Limiting cases [$u_0=k$ and $u_0=-k$]}
\rm(i) For $0<u_0=k$, the fixed point $(0,0)$ is a saddle point due to linear analysis and the nilpotent equilibrium point $(u_0,0)$ is a cusp due to Theorem \ref{kp.r1}. A typical phase
portrait for this case is presented in Figure \ref{kpr04}.
As there is no center in this case, we shall not discuss this case further.\\
\rm(ii) For $u_0=-k<0,$ the fixed points  $(u_0,0)$ and $(k,0)$ are saddle points due to linear analysis and further investigation is required to determine whether $(0,0)$ is a center for the system (\ref{kpr1.17}).  Observe that  $H_{uu}(0,0)=-\beta ku_0>0$ and $\tilde H(0,0)=\beta k^2>0.$ This implies that $(0,0)$ is an isolated minimum and hence $(0,0)$ is a center in view of  Theorem \ref{kp.r2}.

Let us consider all closed trajectories around $(0,0)$. The line  $u=0$ intersects $H(u,y)=h>0$ at two points in the $u-y$ plane for $y=\pm \sqrt{2h}.$ Note that $H(u,y)=0$ corresponds to the center $(0,0).$ Recall that  $(k,0)$ and $(u_0,0)$ are saddles. Further  $ H(k,0)=\frac{\beta k^4}{4}=H(u_0,0)>0.$ Then  $H(u,y)=\frac{\beta k^4}{4}$ corresponds to the heteroclinic orbits connecting $(k,0)$ and $(u_0,0)$. This implies that the trajectories $\Gamma_h,$ $h \in (h_1,h_2)=\left(0,\frac{\beta k^4}{4}\right)$ form the periodic annulus around $(0,0).$ Phase portrait for a choice of the parameters for this case is presented in Figure \ref{kpr004}.

A summary of the classification of all fixed points and the intervals of $h$ such that $\{H(u,y)=h:h\in(h_1,h_2)\}$ form  periodic annuli are presented in Table \ref{tab0}.
\begin{figure}
  \centering
      \includegraphics[width=1.08\linewidth]{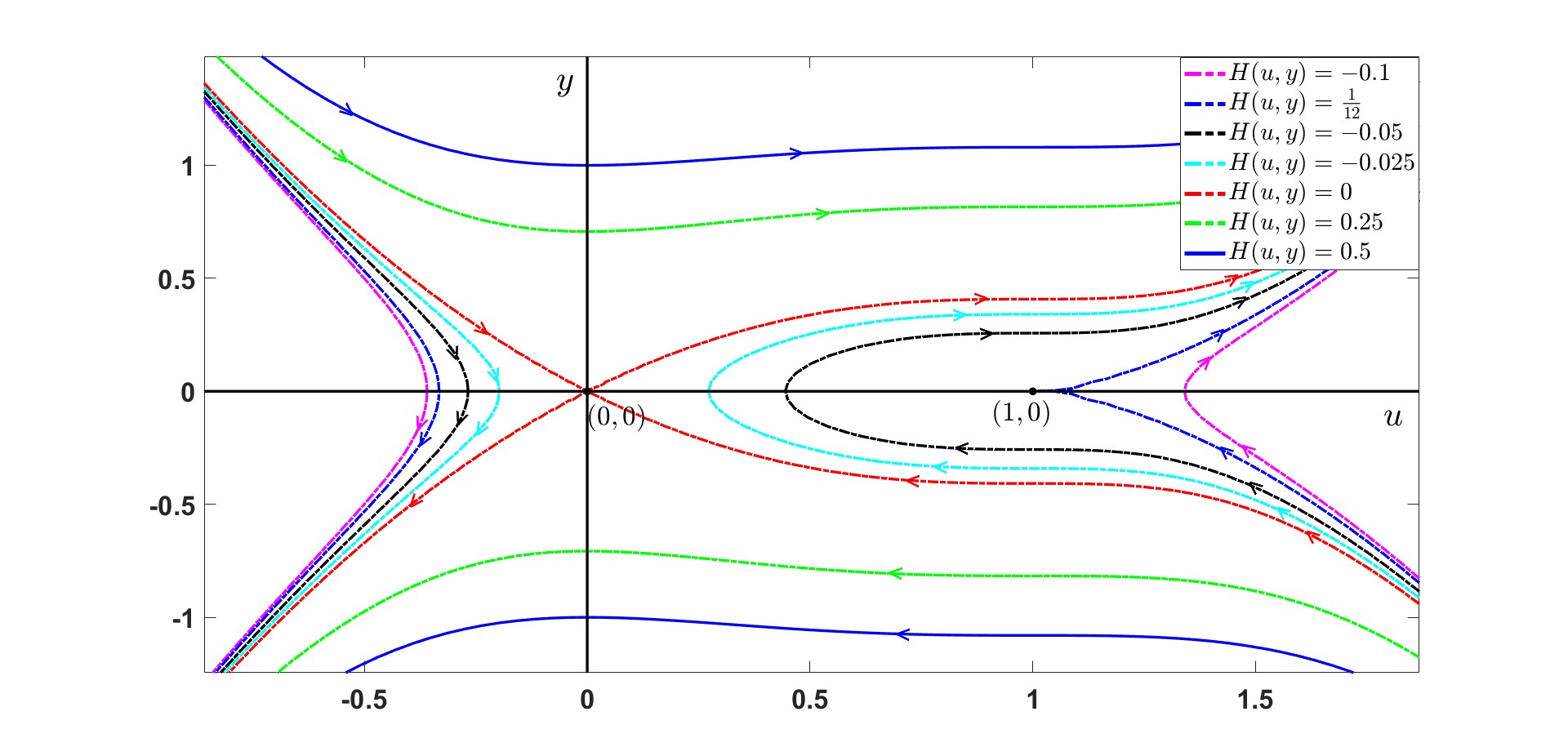}
  \caption{Phase portrait of the system (\ref{kpr1.17}) when $\beta=1,$ $u_0=1,$ and $k=1.$}
  \label{kpr04}
\end{figure}

\begin{figure}
  \centering
      \includegraphics[width=1.08\linewidth]{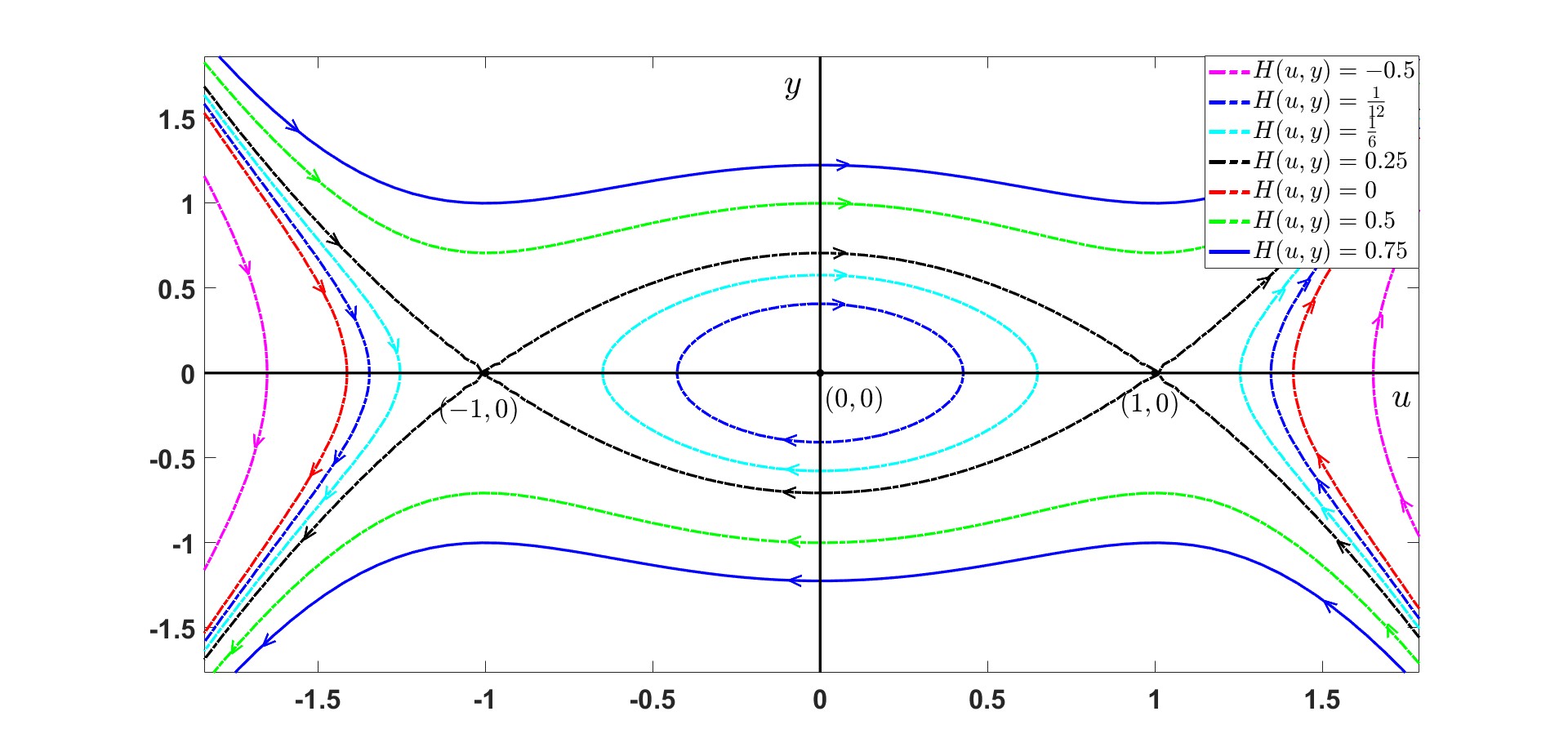}
  \caption{Phase portrait of the system (\ref{kpr1.17}) when $\beta=1,$ $u_0=-1,$ and $k=1.$}
  \label{kpr004}
\end{figure}

\begin{landscape}
\begin{table}
    \centering
\begin{center}
    \begin{tabular}{ |p{1.9cm} | p{1cm}|p{2.8cm}|p{2.4cm} | p{2.2cm} |p{2.2cm} | p{1.2cm} |}
    \hline
    & $ $ & ~~~$0<u_0<k$ &$-k<u_0<0$& \text{~~~~}$u_0<-k$&\text{~~~~}$u_0=-k$ & $u_0=k$\\ \hline
    & {$(0,0)$}& Saddle & Center & Center& Center & Saddle \\ \cline{2-7}
  Nature of fixed point &$(u_0,0)$& Center & Saddle & Saddle & Saddle&Cusp\\ \cline{2-7}
    & $(k,0)$& Saddle & Saddle & Saddle& Saddle & Cusp \\ \hline
 \multirow{3}{*} {Interval of $h$} & \multirow{3}{*}{$(h_1,h_2)$}& $\left( H(u_0,0),H(k,0)\right )$ when $k\leq 2u_0$ &  \multirow{3}{*} {$~\left( 0,H(u_0,0)\right )$} &\multirow{3}{*} {$\left( 0,H(k,0)\right )$} &\multirow{3}{*} {$\left( 0,H(k,0)\right )$} & \multirow{3}{*} {~~~-} \\ \cline{3-3}
    & $ $ & $(H(u_0,0),0)$ when $k> 2u_0$ & & & &\\
    \hline
    \end{tabular}
   
\end{center}
\caption{Classification of fixed points for the system of (\ref{kpr1.17}), and intervals of $h$ for periodic annuli.}
    \label{tab0}
     
\end{table}
\end{landscape}

\begin{table}[!h]
\centering
\begin{center}
    \begin{tabular}{ |p{4cm} | p{4cm}|p{4cm} | p{2cm} |}
    \hline
    ~~~~~ & ~~~~~$w-$ variation & ~~~~~~$u-$ variation\\ \hline
    $~~~0<u_0<k<2u_0$ &~~ $ 0<w_1<w<u_0$; &~~~$u_0<u<k$  \\ 
    ~~~~~~&~~~~~$\Phi(w_1)=\Phi(k)$&~~~~~~\\
    \hline
$~~~0<u_0<k=2u_0$ &~~~ $ 0<w<u_0$ &~~~$u_0<u<2u_0$  \\ \hline  
    $~~~-k<u_0<0$ & $~~~~u_0<w<0$& ~~~~$0<u<u_1<k;$\\
    ~~~~~~&~~~~&~~~~~$\Phi(u_1)=\Phi(u_0)$ \\
    \hline
    $~~~~~~u_0<-k$ & $~~~~u_0<w_2<w<0;$& ~~~~$0<u<k$\\
    ~~~~~& ~~~~~$\Phi(w_2)=\Phi(k)$ &\\
    \hline
    $~~~~~~u_0=-k$ & $~~~-k<w<0$& ~~~~$0<u<k$ \\
    \hline
    \end{tabular}
   \caption{$u-$ and $w-$ variations for different cases for $k>0.$} 
   \label{tab2}
\end{center}
\end{table}
\noindent  Table \ref{tab2} presents the variations of $u$ and $w$. A detailed discussion is presented in subsections 3.1 to 3.3. 
\section{Existence of isolated periodic waves of (\ref{eq3.01})}
In this section we discuss the existence /nonexistence of limit cycles of the perturbed system (\ref{kpr2.5})
using Theorems C and D. We start with a center of the unperturbed system (\ref{kpr1.17}) and then prove the existence / nonexistence of the closed orbits of the perturbed system (\ref{kpr2.5}) depending on the parameters involved for sufficiently small $\epsilon$. That means, our discussion (in this section) will be around the centers of the unperturbed system (\ref{kpr1.17}): 
(i) $(0,0)$ when $u_0<0$ (ii) $(u_0,0)$ when $0<u_0<k.$ We also need the function $\Phi$ to satisfy the condition (a) of Theorem D and the existence of an involution defined around $u=0$ and $u=u_0$ for the cases (i) and (ii), respectively. Graphs of $\Phi$ for the relevant cases (i) and (ii) are presented in Figures \ref{kpinv}-\ref{kpinv5}. One may observe that the condition (a) of Theorem D is satisfied for all the relevant cases (see Figures \ref{kpinv}-\ref{kpinv5}). 

For example, the fixed point $(u_0,0)$ is a center of the system (\ref{kpr1.17}) when $0<u_0<k$. It is easy to verify that $\Phi'(u)(u-u_0)>0$ for $u\in (w_{1},k)\setminus \{u_0\},$ where $w_{1}$ satisfies $\Phi(w_{1})=\Phi(k)$ for  $0<w_{1}<u_0.$ Further there exists an involution $\delta$ defined on $(w_{1},k)$ such that $\Phi(u)=\Phi(\delta(u))$ (see figure \ref{kpinv}).

\begin{figure}
  \centering
      \includegraphics[width=1.08\linewidth]{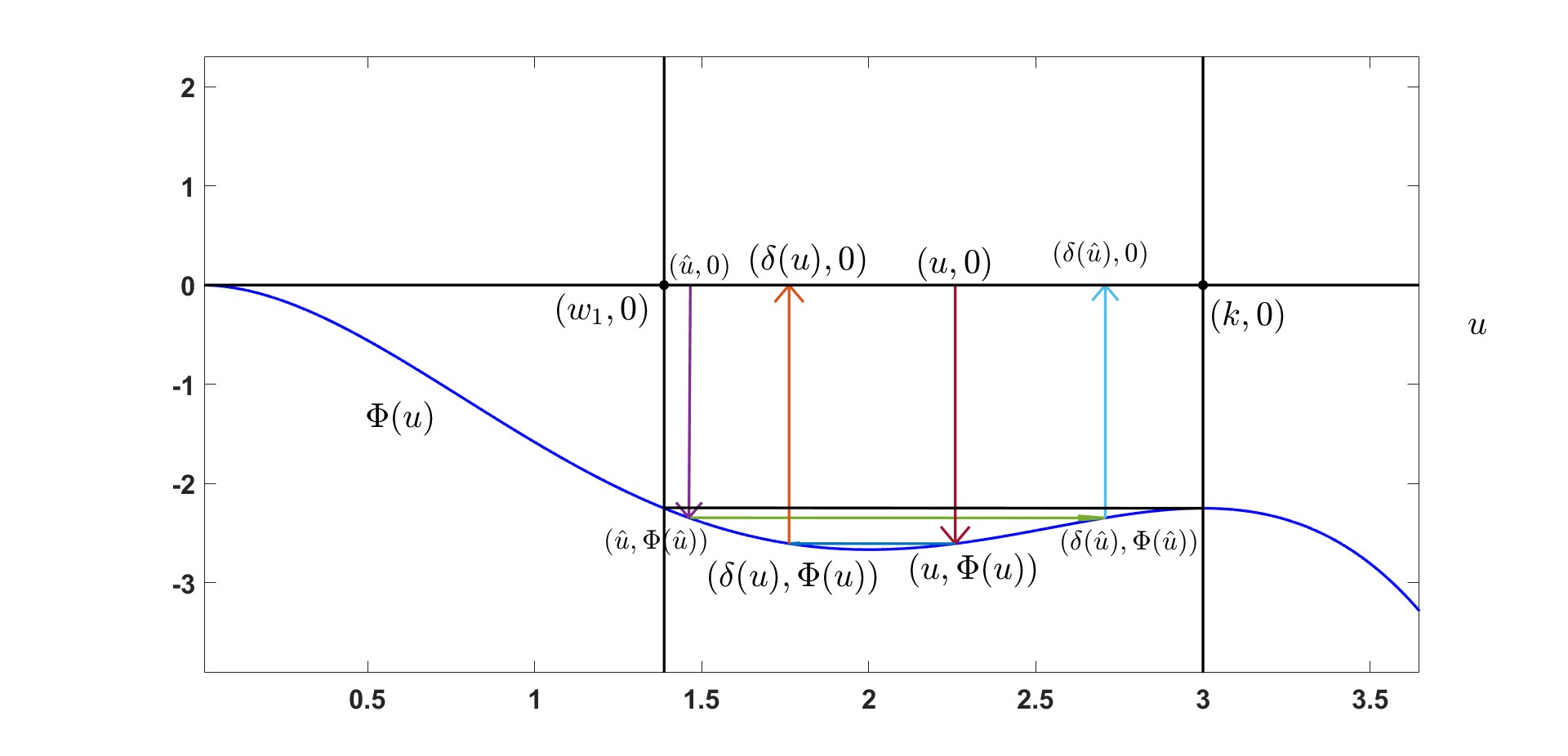}
  \caption{Graph of $\Phi$ with a representation of the involution $\delta$ when $\beta=1,$  $0<u_0<k<2u_0$ ($u_0=2$ and $k=3).$}
  \label{kpinv}
\end{figure}

\begin{figure}
  \centering
      \includegraphics[width=1.08\linewidth]{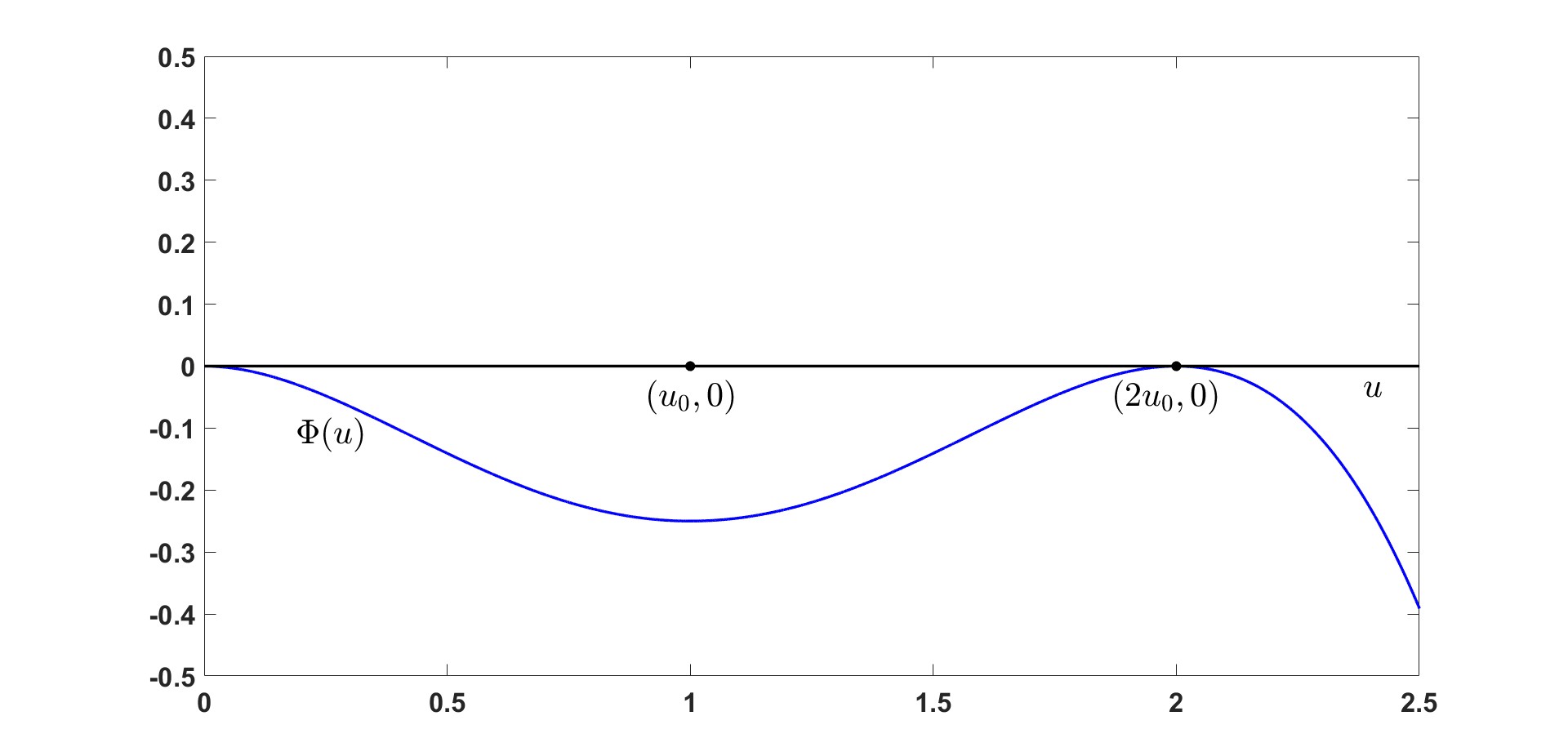}
  \caption{Graph of $\Phi$ when $\beta=1,$ $k=2u_0$ ($u_0=1$ and $k=2$).}
  \label{kpinv2}
\end{figure}

\begin{figure}
  \centering
      \includegraphics[width=1.08\linewidth]{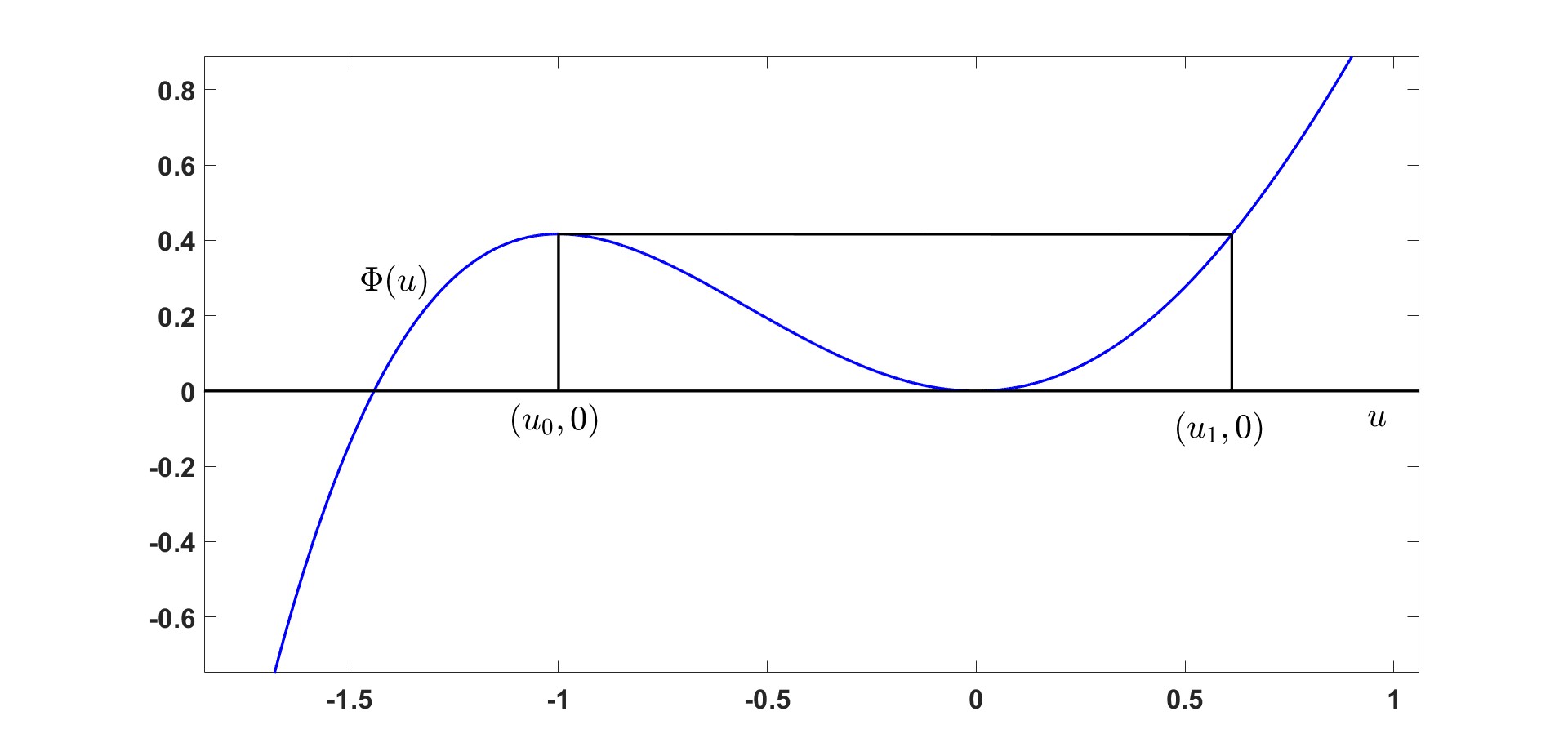}
  \caption{Graph of $\Phi$ when $\beta=1,$ $-k<u_0<0$ ($u_0=-1$ and $k=2$).}
  \label{kpinv3}
\end{figure}

\begin{figure}
  \centering
      \includegraphics[width=1.08\linewidth]{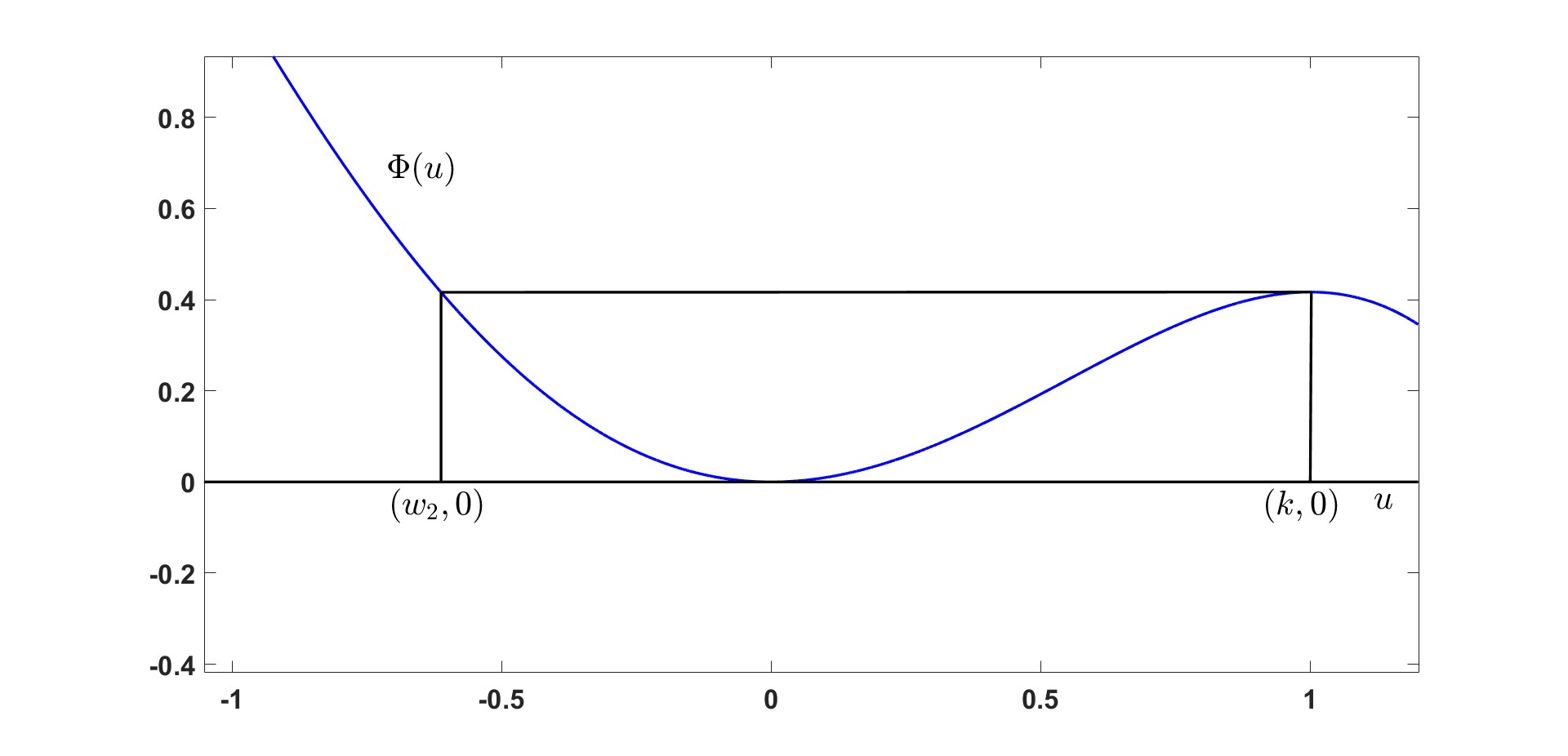}
  \caption{Graph of $\Phi$ when $\beta=1,$ $u_0<-k$ ($u_0=-2$ and $k=1$).}
  \label{kpinv4}
\end{figure}

\begin{figure}
  \centering
      \includegraphics[width=1.08\linewidth]{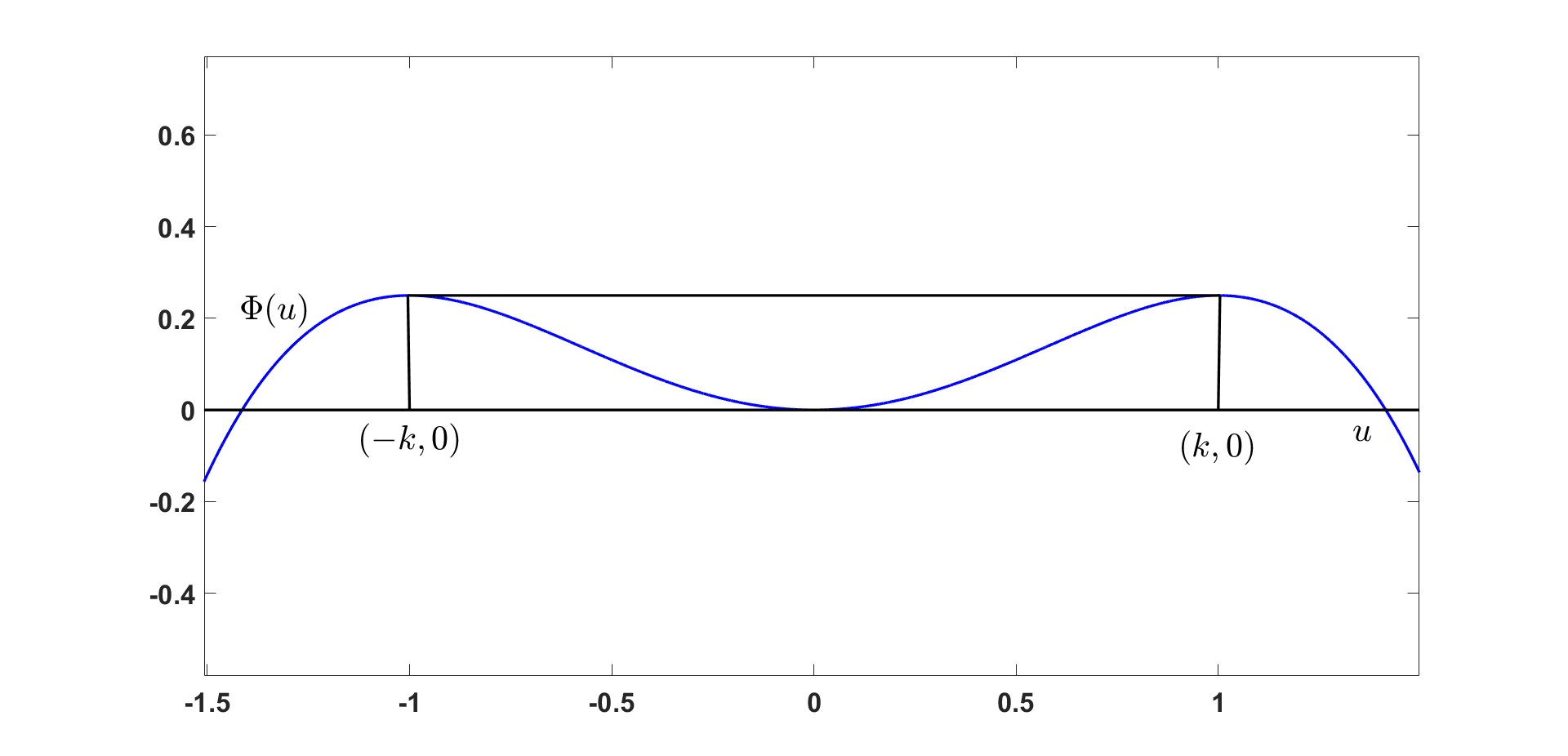}
  \caption{Graph of $\Phi$ when $\beta=1,$ $u_0=-k$ ($u_0=-1$ and $k=1$).}
  \label{kpinv5}
\end{figure}

In this section, we discuss  four cases and the results are summarized in the  following Theorem.
\begin{Theorem}\label{lemma3.01}
Equation (\ref{eq3.01}) has at most one isolated periodic wave solution for the cases :
\begin{itemize}
\item[\rm(i)] $0<u_0<k \leq 2u_0$ for all integers $n \geq 2.$
\item[\rm(ii)] $-k<u_0<0$ for all  integers $n \geq 2.$
\item[\rm(iii)] $u_0<-k$ for all  integers $n \geq 2.$
\item[\rm(iv)] $u_0=-k$ for all even positive integers $n.$
\end{itemize}
\end{Theorem}

One may easily observe  that there exists an involution $\delta$ in each of the caases (i)- (iv) as the relevant fixed point is a  center and condition $(a)$ of the Theorem \ref{thmc} is satisfied. All these cases will be discussed one by one. 

\begin{Lemma}\label{lemmarao1}
     Assume that $n\geq 2$ and $g_n(u,w):=\sum_{i=1}^{n} iw^{i-1}u^{n-i}.$ Then
      \begin{equation}
  g_n(w,u)-g_n(u,w) =
    \begin{cases}
     \sum_{i=1}^{\frac{n}{2}}(n+1-2i)(uw)^{i-1} (u^{n+1-2i}-w^{n+1-2i}), & \text{if $n$ is even }\\ \\
      \sum_{i=1}^{\frac{(n-1)}{2}}(n+1-2i)(uw)^{i-1} (u^{n+1-2i}-w^{n+1-2i}),& \text{if $n$ is odd.} 
    \end{cases}   
    \label{eqnrao3.1}
\end{equation}
    \end{Lemma}
\begin{Proof}
 \textbf{Case 1:} $n$ is even. Now
\begin{align}
  g_n(w,u)-g_n(u,w) &= \sum_{i=1}^{n} iu^{i-1}w^{n-i}-\sum_{i=1}^{n} iw^{i-1}u^{n-i}\nonumber\\
 &= \sum_{i=1}^{n} (n+1-i)u^{n-i}w^{i-1}-\sum_{i=1}^{n} iw^{i-1}u^{n-i}\nonumber\\
&=\sum_{i=1}^{n} (n+1-2i)w^{i-1}u^{n-i}\nonumber\\
&=\sum_{i=1}^{\frac{n}{2}} (n+1-2i)w^{i-1}u^{n-i}+\sum_{i=\frac{n}{2}+1}^{n} (n+1-2i)w^{i-1}u^{n-i}\nonumber\\
&=\sum_{i=1}^{\frac{n}{2}} (n+1-2i)w^{i-1}u^{n-i}-\sum_{i=1}^{\frac{n}{2}} (n+1-2i)u^{i-1}w^{n-i}\nonumber\\
&=\sum_{i=1}^{\frac{n}{2}} (n+1-2i)(w^{i-1}u^{n-i}-u^{i-1}w^{n-i})\nonumber\\
&=\sum_{i=1}^{\frac{n}{2}}(n+1-2i)(uw)^{i-1} (u^{n+1-2i}-w^{n+1-2i}) \nonumber\\
&=\sum_{i=0}^{\frac{n-2}{2}}(n-1-2i)(uw)^{i} (u^{n-1-2i}-w^{n-1-2i}). \label{eqn3.01} 
\end{align}

 \textbf{Case 2:} $n$ is odd. Now
\begin{align}
  g_n(w,u)-g_n(u,w) &= \sum_{i=1}^{n} iu^{i-1}w^{n-i}-\sum_{i=1}^{n} iw^{i-1}u^{n-i}\nonumber\\
 &= \sum_{i=1}^{n} (n+1-i)w^{i-1}u^{n-i}-\sum_{i=1}^{n} iw^{i-1}u^{n-i}\nonumber\\
&=\sum_{i=1}^{\frac{n-1}{2}} (n+1-2i)u^{n-i}w^{i-1}+\sum_{i=\frac{n+3}{2}}^{n} (n+1-2i)u^{n-i}w^{i-1}\nonumber\\
&=\sum_{i=1}^{\frac{n-1}{2}} (n+1-2i)u^{n-i}w^{i-1}-\sum_{i=1}^{\frac{n-1}{2}} (n+1-2i)u^{i-1}w^{n-i}\nonumber\\
&=\sum_{i=1}^{\frac{n-1}{2}} (n+1-2i)w^{i-1}u^{n-i}-\sum_{i=1}^{\frac{n-1}{2}} (n+1-2i)u^{i-1}w^{n-i}\nonumber\\
&=\sum_{i=1}^{\frac{n-1}{2}} (n+1-2i)(w^{i-1}u^{n-i}-u^{i-1}w^{n-i})\nonumber\\
&=\sum_{i=1}^{\frac{n-1}{2}}(n+1-2i)(uw)^{i-1} (u^{n+1-2i}-w^{n+1-2i})\nonumber\\
&=\sum_{i=0}^{\frac{n-3}{2}}(n-1-2i)(uw)^{i} (u^{n-1-2i}-w^{n-1-2i}). \label{3.13}
\end{align}
\end{Proof}\\
Equivalent forms (\ref{eqn3.01}) - (\ref{3.13}) of (\ref{eqnrao3.1}) will be useful later and hence presented here. 
\begin{Lemma}{\label{eq2.2}}
Let $\delta$ be the involution defined accordingly for different cases and $w:=\delta(u)$, $f(u)=(u-u_0)(k-u)$. Then: 
\begin{itemize}
    \item [$(a)$] $uf(u)+wf(w)<0$ for the cases \rm(i) $0<u_0<k<2u_0 $  and \rm(ii) $u_0<-k,$
    \item [$(b)$] $uf(u)+wf(w)>0$ for $-k<u_0<0.$
\end{itemize}
\end{Lemma}
\begin{Proof}
$(a)$\rm(i) Recall that  $\delta$ is the involution such that $\phi(u)=\phi(\delta(u)).$
Then 
\begin{equation}{\label{eq2.7}}
\frac{1}{4}(u^2+w^2)(u+w)-\frac{(k+u_0)}{3}(u^2+uw+w^2)+\frac{ku_0}{2}(u+w)=0. 
\end{equation}
In view of the equation (\ref{eq2.7}), we have $\frac{dw}{du}|_{(u_0,u_0)}=-1$ and $\frac{d^2w}{du^2}|_{(u_0,u_0)}=\frac{4(k-2u_0)}{3u_0(u_0-k)}>0.$ Then $\delta(u)$ is concave upward at $u=u_0$. Note that the equation of the tangent to the curve $w=\delta(u)$ given by (\ref{eq2.7}) at $(u_0,u_0)$ is $u+w=2u_0.$ Let $R$ be the triangular region bounded by $u+w=2u_0,$ $w=u_0$ and $u=k$. The intersection of the curve (\ref{eq2.7}) and the tangent $u+w=2u_0$ is the only one point $(u_0,u_0)$. This clearly indicates that the curve is contained in the region $R$ for $u_0<u<k<2u_0.$

Let us obtain the maximum and minimum values of $uf(u)+wf(w)$ subject to the constraint (\ref{eq2.7}) for all $(u,w)$ in the triangular region $R$. 

Rewriting (\ref{eq2.7}), we have
\begin{eqnarray}
(k+u_0)(u^2+w^2)=(u+w)\left\{\frac{3}{4}(u^2+w^2)+\frac{3ku_0}{2}\right\}-(k+u_0)uw. \label{eqn3.9}
\end{eqnarray}
Now
\begin{eqnarray}
 uf(u)+wf(w) &=& -(u^3+w^3)+(k+u_0)(u^2+w^2)-ku_0(u+w)\nonumber\\
  &=&-(u^3+w^3)+(u+w)\left\{\frac{3}{4}(u^2+w^2)+\frac{3ku_0}{2}\right\}-(k+u_0)uw-ku_0(u+w)\nonumber\\
  &&~~~~~~~~~~~~~~~~~~~~~~~~~~~~~~~~~~~({\rm in \,\, view \,\, of \,\,(\ref{eqn3.9})})\nonumber\\
 &=&-(u+w)(u^2-uw+w^2)+(u+w)\left\{\frac{3}{4}(u^2+w^2)\right\}\nonumber\\
&~~&-(k+u_0)uw+\frac{ku_0}{2}(u+w)\nonumber\\
&=&(u+w)\left\{-\frac{1}{4}(u^2+w^2)\right\}+uw(u+w)+\frac{ku_0}{2}(u+w)-(k+u_0)uw\nonumber\\
&=&-\frac{(k+u_0)}{3}(u^2+uw+w^2)+\frac{ku_0}{2}(u+w)+uw(u+w)\nonumber\\
&~~& +\frac{ku_0}{2}(u+w)-(k+u_0)uw  \text{~~~(using (\ref{eq2.7}))}\nonumber\\
&=& -\frac{(k+u_0)}{3}(u^2+4uw+w^2)+(ku_0+uw)(u+w)\nonumber\\
&=& -\frac{(k+u_0)}{3}(u+w)^2+(ku_0+uw)(u+w)-\frac{2(k+u_0)}{3}uw\nonumber\\
&=&(u+w)\left\{-\frac{(k+u_0)}{3}(u+w)+ku_0+uw\right\}-\frac{2}{3}(k+u_0)uw.\nonumber\\
&:=& F(u,w).
\end{eqnarray}

To prove the main result, it is sufficient to show that $F(u,w)<0$ in $R.$  Note that the function $F(u,w)$ attains its maximum and minimum values on boundary of the triangular region $R$ or at the critical points inside $R$. The corner points of $R$ are $(u_0,u_0),$ $(k,u_0)$ and $(k,2u_0-k).$ Further $F(u_0,u_0)=0,$ $F(k,2u_0-k)=-\frac{2}{3}(2u_0-k)(u_0-k)^2<0$ and $F(k, u_0)=-\frac{1}{3}(u_0+k)(u_0-k)^2<0.$ \\
\textbf{On the boundary $u=k$:}\\
Note that $F(k,w)=-\frac{(k+u_0)}{3}(k^2+4kw+w^2)+(ku_0+kw)(k+w).$ At the critical points,  $\frac{dF}{dw}=0$. This implies that $w=\frac{k(k+u_0)}{2(2k-u_0)}$ and $F\left(k,\frac{k(k+u_0)}{2(2k-u_0)}\right)=\frac{3k^2(k-u_0)^2}{4(u_0-2k)}<0.$\\
\textbf{On the boundary $w=u_0$:}\\
Here $F(u,u_0)=-\frac{(k+u_0)}{3}(u^2+4uu_0+u_0^2)+(ku_0+uu_0)(u+u_0).$ At the critical points, 
 $\frac{dF}{du}=0$. This implies that $u=\frac{u_0(k+u_0)}{2(2u_0-k)}$ and $F\left(\frac{u_0(k+u_0)}{2(2u_0-k)},u_0\right)=-\frac{3u_0^2(k-u_0)^2}{4(2u_0-k)}<0.$\\
\textbf{On the boundary $u+w=2u_0$:}\\
In this case, the only critical point is $(u_0,u_0)$ and  $F(u_0,u_0)=0.$\\
\textbf{In the interior of the region $R$:}\\
Let us find the critical points of $F(u,w)$ in the interior of $R.$ At the critical points of $F(u,w),$ we have $F_u(u,w)=0$ and $F_w(u,w)=0.$ These equations imply that $(u-w)(u+w-\frac{2(k+u_0)}{3})=0.$ The critical points are on the lines $u-w=0$ or $u+w-\frac{2(k+u_0)}{3}=0$ or both.
The line $u+w-\frac{2(k+u_0)}{3}=0$ is parallel to the line $u+w=2u_0$ and stays  left to the line $u+w=2u_0$ as $k<2u_0.$ Therefore no point of the line $u+w-\frac{2(k+u_0)}{3}=0$ belongs to the region $R.$ If we consider the line $u-w=0$,  then only one point $(u_0,u_0)$ of this line belongs to $R$. But $F_u(u_0,u_0)\neq 0$ and $F_w(u_0,u_0)\neq 0.$ Thus the $F(u,w)$ has no critical points interior to the region $R$ and the maximum and minimum values of $F$ are attained on the boundary of $R$. Therefore  $F(u,w)\leq 0$ for all $(u,w) \in R$ and $F(u,w)< 0$ for $0<u_0<u<k<2u_0.$ This, in turn, implies that $uf(u)+wf(w)<0.$

Similarly,  we can prove the other cases also. 
\end{Proof}\\
Define  $$G_{n}(h):=\frac{\oint_{\Gamma_h}\,u^n y\, du}{\oint_{\Gamma_h}\,y\, du},$$ 
where $\Gamma_h$ is the trajectory $H(u,y)=h$ of the unperturbed system from the periodic annulus. 
\subsection{Monotonicity of $G_{n}(h)$ for $0<u_0<k\leq2u_0$}
In this subsection we prove the monotonicity of the function $G_n(h)$. This result will be useful 
for proving the existence of at most one limit cycle around the fixed point $(u_0,0)$ for the perturbed system (\ref{kpr2.5}).  Do note that $(u_0,0)$ is a center for the unperturbed system (\ref{kpr1.17}) when $0<u_0<k\leq2u_0$ and $\Gamma_h$
are the trajectories of the unperturbed system from the periodic annulus around the fixed point $(u_0,0).$ Let us discuss the involution function $\delta(u)$ for the cases $0<u_0<k=2u_0$ and 
$0<u_0<k<2u_0.$\\
Case (i) Let $k=2u_0.$ Here $\Phi \leq 0$ on the interval $[0,k=2u_0]$. Further $\Phi$ has local maxima at $u=0$ and $u=2u_0$ and local minimum at $u=u_0.$ Clearly 
$$
\Phi(0)=\Phi(2u_0)=0, \,\, \Phi(u_0)=-\frac{\beta}{4}u_0^4 <0.
$$
It is easy to see that there is an involution $\delta(u)$ defined for all $u \in (0,2u_0)$.
Further, for $u \in (u_0,2u_0),$ we have $\delta(u)=w \in (0,u_0).$ 
See Figure \ref{kpinv2}. \\
Case (ii) Let  $0<u_0<k <2u_0.$ Here $\Phi(u)\leq 0$ on $[0,k].$ Further $\Phi$ has local maxima at $u=0$ and $u=k$ whereas local minimum at $u=u_0.$ Clearly 
\begin{eqnarray}
\Phi(0)=0, \,\, \Phi(u_0)=-\frac{\beta}{4} u_0^4 <0, \,\, \Phi(k)=\frac{\beta k^3}{12} (k-2u_0)<0.\nonumber\\
\end{eqnarray}
Let $w_1 \in (0,u_0)$ be such that $\Phi(w_1)=\Phi(k).$ It is easy to see that there is an involution function $\delta(u)$ defined for all $u \in (w_1,k)$. Further, for $u\in(u_0,k)$, we have $w=\delta(u) \in (w_1,u_0).$  See Figure \ref{kpinv}. 

Let us prove now some important results.  

\begin{Proposition}\label{lemma3.2}
Let $n \geq 2$ be an integer. Then the function  $G_{n}(h)$ is monotone when $0<u_0<k \leq 2u_0.$ 
\end{Proposition}
\begin{Proof}
 Here we prove the theorem in two cases \rm(i) $0<u_0<k=2u_0,$ and \rm(ii) $0<u_0<k<2u_0.$\\
{\bf{Case 1:}} Here  $0<w<u_0<u<k=2u_0$.

In this case $\Phi(u)=-\frac{\beta}{4} u^2 (u-2u_0)^2$ and the equation $\Phi(u)=\Phi(w)$ implies that $u^2 (u-2u_0)^2=w^2 (w-2u_0)^2.$ This equation, in turn, leads to 
 the involution $\delta(u)=2u_0-u:=w.$ Let
\begin{align}{\label{kp1}}
T_{n}(u):=(n+1)\frac{\int_{\delta(u)}^{u} t^n \,dt }{\int_{\delta(u)}^{u}\,dt }=u^n+u^{n-1}w+...+w^n.
\end{align}
Then
\begin{eqnarray}
 T_{n+1}(u)&=&uT_{n}(u)+w^{n+1}.\nonumber
 \end{eqnarray}
 Differentiating with respect to $u$, we have
 \begin{eqnarray}
  T_{n+1}'(u) &=& uT_{n}'(u)+T_{n}(u)+(n+1)w^{n}\frac{dw}{du}\nonumber\\
 &=&uT_{n}'(u)+T_{n}(u)-(n+1)w^{n} \,\, {\rm because} \,\, \frac{dw}{du}=-1. \label{3.2}
\end{eqnarray}
We  prove that  $$T_{n}'(u)>0, \text{~~where } u_0<u<2u_0.$$
Using the definition of $T_n,$
\begin{eqnarray}
 T_{n}(u)-(n+1)w^{n}&=&\left(\sum_{i=0}^{n} u^{n-i}w^{i}\right)-(n+1)w^{n}\nonumber\\
 &=& \left(\sum_{i=0}^{n-1} u^{n-i}w^{i}\right)-nw^{n}\nonumber\\
 &=&\sum_{i=0}^{n-1} w^{i}(u^{n-i}-w^{n-i}). \label{eqn2.3}
\end{eqnarray}
Equations (\ref{3.2})-(\ref{eqn2.3}) imply that 
\begin{equation} \label{eqn2.4}
    T_{n+1}'(u)=uT_{n}'(u)+ \sum_{i=0}^{n-1} w^{i}(u^{n-i}-w^{n-i}).
\end{equation} Again, by the definition of $T_n,$ 
\begin{eqnarray}  
T_{1}(u)&=&u+w\nonumber\\
 &=& u+(2u_0-u)\,\, {\rm since} \,\, w=2u_0-u \nonumber\\
 &=&2u_0. \label{eqn2.03}
\end{eqnarray}
In view of (\ref{eqn2.4}),  we have
\begin{eqnarray}
 T_{2}'(u)&=&uT_{1}'(u)+ \sum_{i=0}^{0} w^{i}(u^{1-i}-w^{1-i})\nonumber\\
 &=& 0+u-w>0 {\rm ~~since} \,\, u>w. \nonumber
 \end{eqnarray}
 Therefore
 \begin{eqnarray}
 T_{2}'(u)&>&0. \label{eqn02.3}
\end{eqnarray}

Assume that $T_{n}'(u)>0$  for a natural number $n \geq 2.$ Then 
\begin{equation} \label{eqn2.04}
 T_{n+1}'(u)=uT_{n}'(u)+ \sum_{i=0}^{n-1} w^{i}(u^{n-i}-w^{n-i})>0.
\end{equation}

Using mathematical induction, we have $T_{n}'(u)>0$ for all $n\geq2.$ This implies that the function $G_n(h)$ is monotone due to Theorem \ref{thmc}.\\
{\bf{Case 2:}}  Here $0<w_{1}<w<u_0<u<k<2u_0$. Recall that 
$$T'_n(u)=\frac{1}{wf(w)}\{g_n(u,w)uf(u)+g_n(w,u)wf(w)\}, \,\, g_n(u,w):=\sum_{i=1}^{n} iw^{i-1}u^{n-i}.$$
In view of Lemmas \ref{lemmarao1} - \ref{eq2.2} , we have  
$g_{n}(w,u)>g_{n}(u,w)>0$ as $u>w$ and $0<uf(u)<-wf(w),$ respectively.
These two inequalities lead to  $g_n(u,w)uf(u)+g_n(w,u)wf(w)<0$. Note that we have $wf(w)<0$ when $0<w_{1}<w<u_{0}<u<k<2u_{0}.$ Thus we have 
$T_{n}^{\prime}(u)>0$ for $n\geq2$ and, in turn, we have that the function  $G_{n}(h)$ is monotone due to Theorem \ref{thmc}.
\end{Proof}

\noindent\textbf{Remark:} When $n>1$ and $k>2u_0,$ we cannot conclude anything because the function $g_n(u,w)uf(u)+g_n(w,u)wf(w)$ may attain both positive and negative values. 

For $n=2$, $u_0=1$ and $k=2.1,$ $\phi(u)=\phi(w)$ leads to 
\begin{equation}{\label{eq9.1}}
     \frac{1}{4}(u^2+w^2)(u+w)-\frac{3.1}{3}(u^2+uw+w^2)+\frac{2.1}{2}(u+w)=0,
\end{equation}
where $w\in (0,1).$ 

Now when $w=0.1,$ solving equation (\ref{eq9.1}) numerically, we get $u=1.77836.$ For these values of $u$ and $w$, the function $g_2(u,w)uf(u)+g_2(w,u)wf(w)$ has the value $0.22258.$

Again when $w=0.25,$ solving equation (\ref{eq9.1}) numerically, we get $u=1.69015.$
For these values of $u$ and $w$, the function  $g_2(u,w)uf(u)+g_2(w,u)wf(w)$ has  the value $-0.21221.$

Thus we cannot hope to have a result similar to Proposition \ref{lemma3.2} and the analysis used previously cannot be applied for this case. 
\subsection{Monotonicity of $G_{n}(h)$ for the cases $-k<u_0<0$ and $ u_0<-k$}
In this subsection we discuss the monotonicity of the function $G_n(h).$ Note that $\Gamma_h$ are the trajectories around $(0,0)$ of the unperturbed system (\ref{kpr1.17}) from the periodic annulus around $(0,0)$. Let us discuss the relevant involution function $\delta(u)$ for both the cases.\\
(i) Consider  $-k<u_0<0.$ Here 
\begin{eqnarray}
&&\Phi(u_0)=-\frac{\beta}{12}u_0^3 (2k-u_0) >0, \,\,
\Phi(k)=\frac{\beta k^3}{12} (k-2u_0) >0, \nonumber\\
&&\Phi(k)-\Phi(u_0)=\frac{\beta}{12} (k-u_0)^3 (k+u_0)>0.\nonumber
\end{eqnarray}
We have that the function $\Phi(u)$ is non-negative on the interval $[u_0,k]$. Further it has local minimum at $u=0$ and local maxima at $u=u_0$ and $u=k.$  Suppose that $u_1 \in (0,k)$ is such that $\Phi(u_1)=\Phi(u_0).$ Existence of such $u_1$ is assured as $\Phi(u)$ is strictly monotonically increasing on the interval $(0,k)$ and $0<\Phi(u_0)<\Phi(k).$ It is easy to see that there exists an involution $\delta(u)$ defined for all $u \in (u_0,u_1).$ Further, for $u\in (0,u_1),$ we have $w=\delta(u) \in (u_0,0).$ See Figure \ref{kpinv3}.\\
(ii) Consider $u_0 <-k.$ Here $\Phi \geq 0$ on the interval $[u_0,k].$ Further $\Phi$ has local maxima at 
$u=u_0$ and $u=k$ whereas it has local minimum at $u=0.$ Clearly
$$
\Phi(u_0)>0, \,\, \Phi(k)>0, \,\, \Phi(k)<\Phi(u_0).
$$
Suppose that $w_2 \in (u_0,0)$ is such that $\Phi(w_2)=\Phi(k).$ Then the relevant involution function $\delta(u)$ is defined for all $u \in (w_2,k)$. Further, for $u \in (0,k),$ we have $w=\delta(u) \in (w_2,0).$ 
See Figure \ref{kpinv4}.

Having understood clearly the involution function $\delta(u),$ we proceed to prove some important results. 
 \begin{Lemma} {\label{lem01}}
Let $\delta (u)$ be the involution and $f(u)=(u-u_0)(k-u).$ Then \rm(i) $u+\delta(u)<0$ when $-k<u_0<0,$ and \rm(ii) $u+\delta(u)>0$ when $u_0<-k.$
 \end{Lemma}
 \begin{Proof}
Let $w:=\delta(u).$ Note that $u_0<w<0<u<k.$ 
In view of the definition of $\delta (u),$ we have  
 $\Phi(u)=\Phi(w)$ and thus  
\begin{eqnarray}
(u^2+w^2)(u+w)+2ku_0(u+w)=\frac{4(k+u_0)}{3}(u^2+uw+w^2).
\label{eq4.1}
\end{eqnarray}
Also $f(u)+f(w)=(u-u_0)(k-u)+(w-u_0)(k-w)>0$ since $u_0<w<0<u<k.$ \\
 Now
\begin{eqnarray}
 (u+w)(f(u)+f(w)) &=& (u+w)\{(u-u_0)(k-u)+(w-u_0)(k-w)\}\nonumber\\
 &=&-\{(u^2+w^2)(u+w)+2ku_0(u+w)\}+(u+w)^2(k+u_0)\nonumber\\
&=& -\frac{4(k+u_0)}{3}(u^2+uw+w^2)+(u+w)^2(k+u_0) (\text{ using (\ref{eq4.1})})\nonumber\\
&=&-\frac{(k+u_0)}{3} (u-w)^2. \label{eq4.2}
\end{eqnarray}

When $-k<u_0<0,$ we have $f(u)+f(w)>0$ and $u_0+k>0.$ Then $u+w<0$ due to (\ref{eq4.2}).

When $u_0<-k,$ we have $f(u)+f(w)>0$ and $u_0+k<0.$ Then $u+w>0$ due to (\ref{eq4.2}).
 \end{Proof}
 \begin{Lemma} {\label{lem1}}
The function $g_n(u,w)$ takes the form:
\begin{equation}
\begin{split}
   g_n(u,w)&=(u^{n-1}+w^{n-1})+2u^{n-2}w+\sum_{i=1}^{\frac{n-4}{2}}u^{2i}w^{n-1-2i} \\
   &~~~+(u+w)\sum_{i=0}^{\frac{n-4}{2}}(n-1-2i)u^{2i}w^{n-2-2i}, \label{3.17}
    \end{split}
\end{equation}
for all even integers $n \geq 6.$ 
Furthermore  {\rm(i)} $g_n(u,w)<0$  when $-k<u_0<0,$ and {\rm(ii)}  $g_n(w,u)>0$ when $u_0<-k$ for all even integers $n\geq 2$.
\end{Lemma}
\begin{Proof}
Let  $-k <u_0<0.$ Then we have $(u+w)<0$ by Lemma \ref{lem01}. Also note that $u_0<w<0<u<u_{1}<k$. First we  prove that $g_n(u,w)<0$ for  $n=2$ and $n=4$. \\
    For $n=2,$ $$g_2(u,w)=\sum_{i=1}^{2} iw^{i-1}u^{2-i}=u+2w=(u+w)+w<0.$$
    For $n=4,$
    \begin{align*}
  g_4(u,w)=\sum_{i=1}^{4} iw^{i-1}u^{4-i} &= u^3+2wu^2+3w^2u+4w^3\nonumber\\
 &= (u^3+w^3)+2u^2w+3(u+w)w^2.\nonumber\\
\end{align*}
As all the terms  $(u^3+w^3)$, $2u^2w$, and $3(u+w)w^2$ are negative, we have  $g_4(u,w)<0.$

For $n=6,$
\begin{align*}
  g_{6}(u,w) &= u^5+2wu^4+3w^2u^3+4w^3u^2+5w^4u+6w^5\nonumber\\
 &= u^5+w^5+2u^4w+u^2w^3+(u+w)(5w^4+3u^2w^2)\nonumber\\
 &=(u^{6-1}+w^{6-1})+2u^{6-2}w+\sum_{i=1}^{\frac{6-4}{2}}u^{2i}w^{6-1-2i} \nonumber\\
   &~~~+(u+w)\sum_{i=0}^{\frac{6-4}{2}}(6-1-2i)u^{2i}w^{6-2-2i}. 
\end{align*}
Therefore the statement (\ref{3.17}) is true for $n=6.$ 
Assume that the statement (\ref{3.17}) is true for $n=m,$ where $m \geq 6$ is an even integer. Then 
\begin{align}
  g_{m}(u,w) &= (u^{m-1}+w^{m-1})+2u^{m-2}w+\sum_{i=1}^{\frac{m-4}{2}}u^{2i}w^{m-1-2i} \nonumber\\
 &~~~+(u+w)\sum_{i=0}^{\frac{m-4}{2}}(m-1-2i)u^{2i}w^{m-2-2i}.\label{eqn3.1}
\end{align}
Now 
\begin{align}
  &g_{m+2}(u,w)\nonumber\\
  &=\sum_{i=1}^{m+2} iw^{i-1}u^{m+2-i}\nonumber\\
  &= u^2\left ( \sum_{i=1}^{m} iw^{i-1}u^{m-i}\right)+(m+1)uw^m+(m+2)w^{m+1}\nonumber\\
 &=u^2\left(u^{m-1}+w^{m-1}+2u^{m-2}w+\sum_{i=1}^{\frac{m-4}{2}}u^{2i}w^{m-1-2i}\right)\nonumber\\
 &~~~+u^2\left((u+w)\sum_{i=0}^{\frac{m-4}{2}}(m-1-2i)u^{2i}w^{m-2-2i}\right)\nonumber\\
 &~~~+(m+1)uw^m+(m+2)w^{m+1}\,\, ({\rm using } \,\, (\ref{eqn3.1})) \nonumber\\
 &=u^{m+1}+u^2w^{m-1}+2u^{m}w+\sum_{i=1}^{\frac{m-4}{2}}u^{2i+2}w^{m-1-2i}\nonumber\\
 &~~~+(u+w)\sum_{i=0}^{\frac{m-4}{2}}(m-1-2i)u^{2i+2}w^{m-2-2i}+(m+1)uw^m+(m+2)w^{m+1}\nonumber\\
 &=u^{m+1}+u^2w^{m-1}+2u^{m}w+\sum_{i=2}^{\frac{m-2}{2}}u^{2i}w^{m+1-2i}\nonumber\\
 &~~~+(u+w)\sum_{i=1}^{\frac{m-2}{2}}(m+1-2i)u^{2i}w^{m-2i}+(m+1)uw^m+(m+2)w^{m+1}\nonumber\\
 &=u^{m+1}+u^2w^{m-1}+2u^{m}w+\left\{-u^2w^{m-1}+\sum_{i=1}^{\frac{m-2}{2}}u^{2i}w^{m+1-2i}\right\}\nonumber\\
 &~~~+(u+w)\left\{-(m+1)w^{m}+\sum_{i=0}^{\frac{m-2}{2}}(m+1-2i)u^{2i}w^{m-2i}\right\}\nonumber\\
 &~~~+(m+1)uw^m+(m+2)w^{m+1}\nonumber\\
 &=(u^{m+1}+w^{m+1})+2u^{m}w+\sum_{i=1}^{\frac{m-2}{2}}u^{2i}w^{m+1-2i}+(u+w)\sum_{i=0}^{\frac{m-2}{2}}(m+1-2i)u^{2i}w^{m-2i}.
\end{align}
Therefore the statement (\ref{3.17}) is true for $n=m+2$.  
   By mathematical induction,  (\ref{3.17}) is true for all even integers  $n \geq 6$.
   
    Note that  $(u+w)<0$  due to  Lemma \ref{lem01}. Also recall that  $u_0<w<0<u<u_{1}<k$. Then $u^{p}+w^{p}<0$ for odd $p.$ Now  the terms $(u^{n-1}+w^{n-1})$, $2u^{n-2}w$, $\sum_{i=1}^{\frac{n-4}{2}}u^{2i}w^{n-1-2i}$, and $(u+w)\sum_{i=0}^{\frac{n-4}{2}}(n-1-2i)u^{2i}w^{n-2-2i}$ are all negative for all even integers $n \geq 6$. Thus we have that 
     $g_n(u,w)<0$ for all even integers $n\geq 6$ when $-k<u_0<0.$
     
Using similar arguments,  we can prove that $g_n(w,u)>0$ when $u_0<-k$.
\end{Proof}
 \begin{Lemma}{\label{lemma4.4}}
Assume that $m$ is a natural number. Then 
\begin{equation}
\begin{split}
   {\rm(i)}~  g_{4m+2}(w,u)-g_{4m+2}(u,w)&=\sum_{i=0}^{2m}(4m+1-2i)(uw)^{i} (u^{4m+1-2i}-w^{4m+1-2i})\,\, (\text{see (\ref{eqn3.01}) })\\
   &=(u+w)\sum_{i=0}^{m-1}(4m-1-4i)(u^{4m-4i}-w^{4m-4i})(uw)^{2i} \\
   &~~+2\sum_{i=0}^{m-1}(u^{4m+1-4i}-w^{4m+1-4i})(uw)^{2i}+(uw)^{2m}(u-w).\label{3.20}
   \end{split}
   \end{equation}
   \begin{equation}
   \begin{split}
    {\rm(ii)} ~  g_{4m}(w,u)-g_{4m}(u,w)&=\sum_{i=0}^{2m-1}(4m-1-2i)(uw)^{i} (u^{4m-1-2i}-w^{4m-1-2i})\,\, (\text{see (\ref{eqn3.01}) })\\
    &=(u+w)\sum_{i=0}^{m-1}(4m-3-4i)(u^{4m-2-4i}-w^{4m-2-4i})(uw)^{2i} \\
   &~~~+2\sum_{i=0}^{m-1}(u^{4m-1-4i}-w^{4m-1-4i})(uw)^{2i}.\label{3.20abc}~~~~~~~~~~~~~~~~~~~~~~~~~~~~~~~~~~~
    \end{split}
\end{equation}
Furthermore $g_n(w,u)>g_n(u,w)$ for all even positive integers when $-k<u_0<0$ and $u_0<-k.$
\end{Lemma}
\begin{Proof}
\rm(i) We prove (\ref{3.20}) by induction on $m.$
 For $m=1,$ we have $n=6.$ Note that 
\begin{align*}
  g_6(w,u)-g_6(u,w)&=\sum_{i=0}^{2}(5-2i)(uw)^{i} (u^{5-2i}-w^{5-2i})\,\, ({\rm see } \,\, (\ref{eqn3.01} )) \nonumber\\
 &= 5(u^5-w^5)+3uw(u^3-w^3)+(uw)^2(u-w)\nonumber\\
 &=2(u^5-w^5)+3\{u^5-w^5+uw(u^3-w^3)\}+(uw)^2(u-w) \nonumber\\
&=2(u^5-w^5)+3(u+w)(u^4-w^4)+(uw)^2(u-w)\nonumber\\
   &=(u+w)\sum_{i=0}^{1-1}(3-4i)(u^{4-4i}-w^{4-4i})(uw)^{2i} \nonumber\\
   &~~~+2\sum_{i=0}^{1-1}(u^{5-4i}-w^{5-4i})(uw)^{2i}+(uw)^{2}(u-w).
\end{align*}
 Therefore the statement (\ref{3.20}) is true for $n=6,$ i.e. $m=1.$\\
 Assume that the equation (\ref{3.20}) is true for $n=2(2m+1)$ for a natural number $m.$ Then, by (\ref{3.20}),
\begin{align}
  g_{4m+2}(w,u)-g_{4m+2}(u,w)&=\sum_{i=0}^{2m}(4m+1-2i)(uw)^{i} (u^{4m+1-2i}-w^{4m+1-2i})\nonumber\\
   &=(u+w)\sum_{i=0}^{m-1}(4m-1-4i)(u^{4m-4i}-w^{4m-4i})(uw)^{2i} \nonumber\\
   &+2\sum_{i=0}^{m-1}(u^{4m+1-4i}-w^{4m+1-4i})(uw)^{2i}+(uw)^{2m}(u-w).  \label{eqn3.03}
\end{align}
 Let us prove (\ref{3.20}) for $n=2(2(m+1)+1)$ i.e $n=4m+6.$ Now
 \begin{align}
 g_{4m+6}(w,u)-g_{4m+6}(u,w)&=\sum_{i=0}^{2m+2}(4m+5-2i)(uw)^{i} (u^{4m+5-2i}-w^{4m+5-2i})\,\, ({\rm using } \,\, (\ref{eqn3.01} )) \nonumber\\
 &=\sum_{i=-2}^{2m}(4m+1-2i)(uw)^{i+2} (u^{4m+1-2i}-w^{4m+1-2i})\nonumber\\
 &=(uw)^2\sum_{i=0}^{2m}(4m+1-2i)(uw)^{i} (u^{4m+1-2i}-w^{4m+1-2i})\nonumber\\
 &~~~+(4m+5)(u^{4m+5}-w^{4m+5})+(4m+3)(uw)(u^{4m+3}-w^{4m+3})\nonumber\\
 &=(u+w)\sum_{i=0}^{m-1}(4m-1-4i)(u^{4m-4i}-w^{4m-4i})(uw)^{2i+2} \nonumber\\
   &~~~+2\sum_{i=0}^{m-1}(u^{4m+1-4i}-w^{4m+1-4i})(uw)^{2i+2}+(uw)^{2m+2}(u-w) \nonumber\\
   &~~~+2(u^{4m+5}-w^{4m+5})+(4m+3)(u+w)(u^{4m+4}-w^{4m+4})\,\, ({\rm using } \,\, (\ref{eqn3.03}))\nonumber\\  
  &=(u+w)\sum_{i=1}^{m}(4m+3-4i)(u^{4m+4-4i}-w^{4m+4-4i})(uw)^{2i} \nonumber\\
   &~~~+2\sum_{i=1}^{m}(u^{4m+5-4i}-w^{4m+5-4i})(uw)^{2i}+(uw)^{2m+2}(u-w) \nonumber\\
   &~~~+2(u^{4m+5}-w^{4m+5})+(4m+3)(u+w)(u^{4m+4}-w^{4m+4}) \nonumber\\  
  &=(u+w)\sum_{i=0}^{m}(4m+3-4i)(u^{4m+4-4i}-w^{4m+4-4i})(uw)^{2i} \nonumber\\
   &~~~+2\sum_{i=0}^{m}(u^{4m+5-4i}-w^{4m+5-4i})(uw)^{2i}+(uw)^{2m+2}(u-w).
\end{align}
 This implies that the equation (\ref{3.20}) is true for $n=4m+6.$ Therefore, by Mathematical Induction, the equation (\ref{3.20}) is true for all $n=2(2m+1),$ where $m$ is any natural number.

 \rm(ii) Using similar method, we can prove that (\ref{3.20abc}) is true for all $n=2(2m),$ where $m$ is any natural number.

 When $-k<u_0<0,$ we have $u+w<0$ due to Lemma \ref{lem01}. Then $u^{2m}-w^{2m}<0$ and $u^{2m-1}-w^{2m-1}>0.$ This implies that $g_n(w,u)>g_n(u,w)$ for all even positive integer $n \geq 4.$ 
Note that $g_2(w,u)-g_2(u,w)=w+2u-(u+2w)=u-w>0.$ Therefore  $g_n(w,u)>g_n(u,w)$ for all even positive integers.

  Similarly,  we can prove that   $g_n(w,u)>g_n(u,w)$ when $u_0<-k.$
 \end{Proof}
\begin{Lemma} \label{lem3.11}
The function $g_n(u,w)$ can be written in the form:
\begin{equation}
\begin{split}
   g_n(u,w)&=(u^{n-1}+w^{n-1})+\sum_{i=1}^{\frac{n-3}{2}}u^{2i}w^{n-1-2i}\\
   &~~~+(u+w)\sum_{i=0}^{\frac{n-3}{2}}(n-1-2i)u^{2i}w^{n-2-2i}, \label{3.24a} 
    \end{split}
\end{equation}
for all odd integers $n \geq 5$.
Furthermore  \rm(i) $g_n(u,w)>0$ when $-k<u_0<0,$ and \rm(ii) $g_n(w,u)>0$ when $u_0<-k$ for all odd integers $n \geq 3.$
\end{Lemma}
\begin{Proof}
 For $n=3,$ 
 $$g_3(u,w)=\sum_{i=1}^{3} iw^{i-1}u^{3-i}=u^2+2uw+3w^2=(u+w)^2+2w^2>0.$$ 

For $n=5,$
\begin{align*}
  g_5(u,w)=\sum_{i=1}^{5} iw^{i-1}u^{5-i} &= u^4+2wu^3+3w^2u^2+4w^3u+5w^4\nonumber\\
 &= u^4+w^4+u^2w^2+(u+w)(4w^3+2u^2w)\nonumber\\
 &=(u^{5-1}+w^{5-1})+\sum_{i=1}^{\frac{5-3}{2}}u^{2i}w^{5-1-2i} \nonumber\\
   &~~~+(u+w)\sum_{i=0}^{\frac{5-3}{2}}(5-1-2i)u^{2i}w^{5-2-2i}.
\end{align*}

Therefore the equation (\ref{3.24a}) is true for $n=5.$

Assume that (\ref{3.24a}) is true for $n=m,$ where $m \geq 5$ is an odd integer. Then
\begin{align}
  g_{m}(u,w)&=(u^{m-1}+w^{m-1})+\sum_{i=1}^{\frac{m-3}{2}}u^{2i}w^{m-1-2i} \nonumber\\
   &+(u+w)\sum_{i=0}^{\frac{m-3}{2}}(m-1-2i)u^{2i}w^{m-2-2i}.\label{eqn3.3}
\end{align}
Now 
\begin{align*}
  g_{m+2}(u,w) &= u^2\left ( \sum_{i=1}^{m} iw^{i-1}u^{m-i}\right)+(m+1)uw^m+(m+2)w^{m+1}\nonumber\\
 &=u^2\left(u^{m-1}+w^{m-1}+\sum_{i=1}^{\frac{m-3}{2}}u^{2i}w^{m-1-2i}\right)\nonumber\\
 &~~~+u^2\left((u+w)\sum_{i=0}^{\frac{m-3}{2}}(m-1-2i)u^{2i}w^{m-2-2i}\right)\nonumber\\
 &~~~+(m+1)uw^m+(m+2)w^{m+1}\,\, ({\rm using } \,\, (\ref{eqn3.3})) \nonumber\\
 &=u^{m+1}+u^2w^{m-1}+\sum_{i=1}^{\frac{m-3}{2}}u^{2i+2}w^{m-1-2i}\nonumber\\
 &~~~+(u+w)\sum_{i=0}^{\frac{m-3}{2}}(m-1-2i)u^{2i+2}w^{m-2-2i}+(m+1)uw^m+(m+2)w^{m+1}\nonumber\\
 &=u^{m+1}+u^2w^{m-1}+\sum_{i=2}^{\frac{m-1}{2}}u^{2i}w^{m+1-2i}\nonumber\\
 &~~~+(u+w)\sum_{i=1}^{\frac{m-1}{2}}(m+1-2i)u^{2i}w^{m-2i}+(m+1)uw^m+(m+2)w^{m+1}\nonumber\\
 &=u^{m+1}+u^2w^{m-1}+\left\{-u^2w^{m-1}+\sum_{i=1}^{\frac{m-1}{2}}u^{2i}w^{m+1-2i}\right\}\nonumber\\
 &~~~+(u+w)\left\{-(m+1)w^{m}+\sum_{i=0}^{\frac{m-1}{2}}(m+1-2i)u^{2i}w^{m-2i}\right\}\nonumber\\
 &~~~+(m+1)uw^m+(m+2)w^{m+1}\nonumber\\
 &=(u^{m+1}+w^{m+1})+\sum_{i=1}^{\frac{m-1}{2}}u^{2i}w^{m+1-2i}+(u+w)\sum_{i=0}^{\frac{m-1}{2}}(m+1-2i)u^{2i}w^{m-2i}.
\end{align*}
Therefore the equation (\ref{3.24a}) is true for $n=m+2$ and hence,  
   by mathematical induction, the equation (\ref{3.24a}) is true for all odd integers $n \geq 5$. 
   
   Let $-k<u_0<0.$ Then we have $(u+w)<0$ by Lemma \ref{lem01}. Recall that  $u_0<w<0<u<u_{1}<k.$ Then $(u^{n-1}+w^{n-1})$, $\sum_{i=1}^{\frac{n-3}{2}}u^{2i}w^{n-1-2i}$, and $(u+w)\sum_{i=0}^{\frac{n-3}{2}}(n-1-2i)u^{2i}w^{n-2-2i}$ are positive. This implies that $g_n(u,w)>0$ for $n \geq 5$. We already proved $g_3(u,w)>0$ and thus we have that $g_n(u,w)>0$ for all odd  integers $n \geq 3.$
   
  Similarly,  we can prove that $g_n(w,u)>0$ when $u_0<-k.$
\end{Proof}

\begin{Lemma} \label{lem3.9}
Assume that $m$ is a natural number. Then 
\begin{equation}
\begin{split}
   {\rm(i)}~  g_{4m+1}(w,u)-g_{4m+1}(u,w)&=\sum_{i=0}^{2m-1}(4m-2i)(uw)^{i} (u^{4m-2i}-w^{4m-2i}) \text{~~(using (\ref{3.13}))}\\
   &=(u+w)\sum_{i=0}^{m-1}(4m-2-4i)(u^{4m-1-4i}-w^{4m-1-4i})(uw)^{2i} \\
   &~~~+2\sum_{i=0}^{m-1}(u^{4m-4i}-w^{4m-4i})(uw)^{2i}.
   \end{split}
   \end{equation}
   \begin{equation}
\begin{split}  
   {\rm(ii)}~  g_{4m+3}(w,u)-g_{4m+3}(u,w)&=\sum_{i=0}^{2m}(4m+2-2i)(uw)^{i} (u^{4m+2-2i}-w^{4m+2-2i})\text{~~(using (\ref{3.13}))}\\
   &=(u+w)\sum_{i=0}^{m-1}(4m-4i)(u^{4m+1-4i}-w^{4m+1-4i})(uw)^{2i} \\
   &~~~+2\sum_{i=0}^{m}(u^{4m+2-4i}-w^{4m+2-4i})(uw)^{2i}.~~~~~~~~~~~~~~~~~~~~~\\
    \end{split}
\end{equation}
Furthermore $(a)$ $g_n(u,w)>g_n(w,u)$ when $-k<u_0<0,$ and $(b)$ $g_n(u,w)<g_n(w,u)$ when $u_0<-k$ for all odd integers $n \geq 3.$
\end{Lemma}
\begin{Proof}
Following closely the proof of Lemma \ref{lemma4.4}, we can prove this lemma.
\end{Proof}
\begin{Proposition}\label{lemma4.2}
Let $n \geq 2$ be an integer. Then the function $G_{n}(h)$ is monotone when \rm(i) $-k<u_0<0,$ and \rm(ii) $u_0<-k.$ 
\end{Proposition}
\begin{Proof}
We prove the proposition in two cases \rm(a) $n$ is even, \rm(b) $n(>1)$ is odd.
\\
\textbf{Case \rm(a):} Let us assume that  $-k<u_0<0.$
In view of Lemma \ref{lem1}, we have $g_n(u,w)<0$ for all even positive integers. Further we have $uf(u)>-wf(w)>0$ due to Lemma \ref{eq2.2}. Further, in view of Lemma \ref{lemma4.4}, $g_n(w,u)>g_n(u,w)$ (i. e., $-g_n(u,w)>-g_n(w,u))$ for all even positive even integers.   These two inequalities imply that $g_n(u,w)uf(u)+g_n(w,u)wf(w)<0.$ Therefore we have 
$$T'_{n}(u)=\frac{1}{wf(w)}\{g_n(u,w)uf(u)+g_n(w,u)wf(w)\}>0$$ 
for all even positive integers as $f(w)>0$ and $u_0<w<0<u<u_{1}<k.$ This, in turn, implies that the function $G_{n}(h)$ is monotone on $\left(0,\frac{\beta u_{0}^3}{12}(u_0-2k)\right)$ for all even positive integers  due to  Theorem \ref{thmc}.

 Now let us assume that $u_0<-k.$ We have  $g_n(w,u)>g_n(u,w)$ for all even positive integers in view of Lemma \ref{lemma4.4}. Furthermore $g_n(w,u)>0$ and $0<uf(u)<-wf(w)$ for all even positive integers due to Lemmas \ref{lem1} and \ref{eq2.2}. These inequalities imply that $g_n(u,w)uf(u)+g_n(w,u)wf(w)<0.$ Also we have $wf(w)<0$ when $u_{0}<w_{2}<w<0<u<k.$ Therefore we have $T'_{n}(u)>0$ for all even positive integers. This, in turn, implies that the function  $G_{n}(h)$ is monotone on $\left(0,\frac{\beta k^3}{12}(k-2u_0)\right)$ for all even positive integers due to  Theorem \ref{thmc}.
\\
\textbf{Case \rm(b):}
  Let us assume that  $-k<u_0<0.$ We have  $g_n(w,u)<g_n(u,w)$ for all odd integers $n \geq 3$ due to Lemma \ref{lem3.9}. Furthermore, by Lemma \ref{lem3.11}, $g_n(u,w)>0$ for all odd integers $n \geq 3$. We also have $uf(u)>-wf(w)>0$ in view of Lemma \ref{eq2.2}. These inequalities imply that $g_n(u,w)uf(u)+g_n(w,u)wf(w)>0.$ Note that we have $wf(w)<0$ when $u_0<w<0<u<u_{1}<k.$ Therefore we have $T'_{n}(u)<0$ for all odd integers $n \geq 3$. This, in turn, implies that the function  $G_{n}(h)$ is monotone on $\left(0,\frac{\beta u_{0}^3}{12}(u_0-2k)\right)$ for all odd integers $n \geq 3$ due to Theorem \ref{thmc}. 
  
  Now let us assume that $u_0<-k.$ In view of Lemma \ref{lem3.9}, $g_n(w,u)>g_n(u,w)$ for all odd integers $n \geq 3$. Furthermore $g_n(w,u)>0$ for all odd integers $n \geq 3$ and  $0<uf(u)<-wf(w)$ due to Lemmas \ref{lem3.11} and \ref{eq2.2}. These inequalities imply that $g_n(u,w)uf(u)+g_n(w,u)wf(w)<0.$ Recall that $wf(w)<0$ when $u_0<w_{2}<w<0<u<k.$ Therefore we have $T'_{n}(u)>0$ for all odd integers $n \geq 3.$ 
  This, in turn, implies that the function $G_{n}(h)$ is monotone on $\left(0,\frac{\beta k^3}{12}(k-2u_0)\right)$ for all odd integers $n \geq 3$ due to Theorem \ref{thmc}. This completes the proof.
\end{Proof}
\subsection{Monotonicity of $G_{n}(h)$ for $u_0=-k$}
In this subsection we deal with the  monotonicity of the function $G_n(h).$ Note that $\Gamma_h$ are the trajectories  from the periodic annulus around (0,0) of the unperturbed system (\ref{kpr1.17}) when $u_0=-k.$ 

The function $\Phi(u)$ is nonnegative on $[-k,k]$. Further 
$\Phi(u_0=-k)=\Phi(k)=\frac{\beta k^4}{4} $
and $\Phi^\prime(u) <0$ on the interval $(-k,0)$ and $\Phi^\prime(u) >0$ on $(0, k).$ It is easy to see that 
there exists an involution $\delta(u)$ defined for all $u \in (-k,k)$. Further, for $u \in (0,k),$ we have 
$w=\delta(u) \in (-k,0).$ See Figure \ref{kpinv5}.
\begin{Proposition}{\label{thm2.5}}
Assume that   $\beta>0$, $k>0$ and $u_0=-k$.  Then, for all even integers $n$, the ratio $G_{n}(h)$ is monotone when $h\in\left(0,\frac{\beta u_{0}^4}{4}\right)$.
\end{Proposition}
\begin{proof}
Here we have $\Phi'(u)(u-0)>0$ for $u\in(u_0,k) \setminus\{0\}$. It is easy to see that the involution $\delta(u)=-u$ defined on $(u_0,k).$ Here $H(0,0)=0$ and $H(u_0,0)=H(k,0)=\frac{\beta u_{0}^4}{4}.$ Further the family of trajectories $\left\{\Gamma_h:H(u,y)=h,h\in(0,\frac{\beta u_{0}^4}{4})\right\}$ of the unperturbed system (\ref{kpr1.17}) forms the periodic annulus. For $n$ even integer, 
\begin{align*}
T_{n}(u)=(n+1)\frac{\int_{\delta(u)}^{u} t^n \,dt }{\int_{\delta(u)}^{u}\,dt }=(n+1)\frac{\int_{-u}^{u} t^n \,dt }{\int_{-u}^{u}\,dt }=u^{n}.
\end{align*}
Then $T'_{n}(u)=nu^{n-1}>0$ for $u\in(0,k)$. This implies that the function $G_{n}(h)$ is monotone for $h\in\left(0,\frac{\beta u_{0}^4}{4}\right)$ due to Theorem \ref{thmc}.
\end{proof}
\noindent \textbf{Note:} When $u_0=-k$,  the fixed point $(0,0)$ is the center of the system (\ref{kpr1.17}). Further all the closed orbits $\Gamma_h$ of the system (\ref{kpr1.17}) surrounding the point $(0,0)$  are symmetric and the direction of each $\Gamma_h$ is clockwise. Thus $ A_0(h)>0$ due to Green's Theorem. When $n$ is an odd positive integer, $A_n(h)=\oint_{\Gamma_h}\,u^ny\, du=2\int_{-u}^{u} u^n y \,du \equiv 0$. Then 
Abelian integral $A(h)$ of the system (\ref{kpr2.5}) is
\begin{align}
   &A(h):=\alpha_0 A_0(h)+\alpha_n A_n(h) \nonumber \\
   &~~~~~~~=\alpha_0 A_0(h).
\end{align}
Therefore $A(h)>0$ (or $<0$) if $\alpha_0>0$ (or $<0$). Thus $A(h)$ has no zeros and hence the system (\ref{kpr2.5}) has no limit cycles in view of Theorem \ref{lemma3.1}.\\ \\
\subsection{Proof of Theorem \ref{lemma3.01} }
In this subsection we prove Theorem \ref{lemma3.01} (i) and all other cases \rm(ii), \rm(iii), and \rm(iv) can be proved using similar arguments.

In section 2, for $0<u_0<k\leq 2u_{0},$ we observed that $\Gamma_h$ tends to the center $(u_0,0)$ as $h\to\frac{\beta u_{0}^3}{12}(u_0-2k).$ When $k<2u_0,$ $\Gamma_h$ tends to a homoclinic orbit connecting $(k,0)$ as $h\to \frac{\beta k^3}{12}(k-2u_0).$ When $k=2u_0,$ $\Gamma_h$ tends to the heteroclinic orbits connecting $(0,0)$ and $(k,0)$ as $h\to 0.$

Let us consider the case $k<2u_0.$ Now the function  $\frac{A_n(h)}{A_0(h)}$ is monotone for $h\in\left(\frac{\beta u_{0}^3}{12}(u_0-2k),\frac{\beta k^3}{12}(k-2u_0)\right):=I$ due to Proposition \ref{lemma3.2}. Further we have $A(h)=A_0(h)\left(\alpha_0+\alpha_n \frac{A_n(h)}{A_0(h)}\right).$ If $\frac{A_n(h)}{A_0(h)} \neq -\frac{\alpha_0}{\alpha_n}$ for $h\in I$, then $A(h)\neq 0$ for $h\in I.$ Therefore the system (\ref{kpr2.5}) has no limit cycles in view of Theorem \ref{lemma3.1}.

If there exist two real numbers $\alpha_0$ and $\alpha_n$ such that $\frac{A_n(\tilde{h})}{A_0(\tilde{h})}=-\frac{\alpha_0}{\alpha_n}$ for some $\tilde{h}\in I,$ then $A(h)$ has exactly one zero in the interval $I$ as the function $\frac{A_n(h)}{A_0(h)}$ is monotone. Theorem \ref{lemma3.1} implies that the perturbed system (\ref{kpr2.5})  has a unique limit cycle and hence  the equation  (\ref{eq3.01}) has a unique isolated periodic wave solution. 
\section{Interval of the ratio $\frac{\alpha_0}{\alpha_n}$ for the existence of limit cycle and numerical study}
Let us briefly explain the procedure to find the interval of $\frac{\alpha_0}{\alpha_n}$ for the existence of limit cycles for the perturbed system (\ref{kpr2.5}). First 
we identify all the parametric ranges for which there exist limit cycles. Fix one of the parametric ranges.
For that parametric range,  identify the periodic annulus. That is, find the interval $(h_1,h_2)$ such that the trajectories $\Gamma_h:H(u,y)=h, h\in(h_1,h_2)$ of the unperturbed system (\ref{kpr1.17}) form the periodic annulus. We have already shown that $\frac{A_n(h)}{A_0(h)}$ is monotone on the interval $(h_1,h_2)$.  For any positive integer $n>1,$ there exist two real numbers $K_{n,h_1}$ and $K_{n,h_2}$ such that $\frac{A_n(h)}{A_0(h)} \to K_{n,h_1}$ as $h \to h_{1}$, and $\frac{A_n(h)}{A_0(h)} \to K_{n,h_2}$ as $h \to h_{2}.$ Therefore  $\frac{A_n(h)}{A_0(h)}\in (K_{n,h_1},K_{n,h_2})$ or $\frac{A_n(h)}{A_0(h)}\in (K_{n,h_2},K_{n,h_1})$ for all $h\in (h_1,h_2)$ depending on whether $K_{n,h_1}<K_{n,h_2}$ or $K_{n,h_1}>K_{n,h_2}$. Assume that $\frac{A_n(h)}{A_0(h)}\in (K_{n,h_1},K_{n,h_2})$. If there exists a limit cycle, then $-\frac{\alpha_0}{\alpha_n}=\frac{A_n(h)}{A_0(h)}\in (K_{n,h_1},K_{n,h_2}).$ This, in turn,  implies that $\frac{\alpha_0}{\alpha_n}\in (-K_{n,h_2},-K_{n,h_1}).$ 

Let us explain the procedure with  one particular example. Consider  $n=2,$ $u_0=1,$ $\beta=1$, and $k=2.$ In this case the set of trajectories \{$\Gamma_h: H(u,y)=h, h\in(-\frac{1}{4},0)\}$ forms the periodic annulus. If $h \to -\frac{1}{4},$ then $\frac{A_2}{A_0} \to 1,$ and If $h \to 0,$ then $\frac{A_2}{A_0} \to \frac{6}{5}.$ Clearly, for the existence of limit cycle, we need $-\frac{\alpha_0}{\alpha_2}=\frac{A_2(h)}{A_0(h)}\in (1,\frac{6}{5}).$ This indicates that $\frac{\alpha_0}{\alpha_2}\in \left(-\frac{6}{5},-1\right).$ More generally,  we can get $\frac{\alpha_0}{\alpha_n} \in \left(-\frac{3(2^{n+1})}{(n+2)(n+3)},-1\right)$ for the existence of limit cycle for any positive integer $n.$ 

Let us consider $n=2,$ $\beta=1$, $u_0=1$, $k=2,$ $\epsilon=0.1$ and the initial values $(u,y)=(0.5,0)$. Clearly $H(0.5,0)=-\frac{9}{64}:=h^*$ and the corresponding trajectory of the unperturbed system (\ref{kpr1.17}) is $\Gamma_{h^*}: H(u,y)=\frac{y^2}{2}+\left (-\frac{u^4}{4}+u^3-u^2\right)=-\frac{9}{64}$. 
Note that $\frac{A_2(h^*)}{A_0(h^*)}=\frac{\oint_{\Gamma_{h^*}}\,u^2y\, du}{\oint_{\Gamma_{h^*}}\,y\, du}$. 
By direct computation, we get  $-\frac{\alpha_0}{\alpha_2}=\frac{A_2(h^*)}{A_0(h^*)}=1.06133$. Let us choose $\alpha_0=-1.06133$ and $\alpha_2=1$. This clearly gives $A(h^*)=0.$ \\ 

Our numerical study shows that:
\begin{itemize}
    \item [\rm(i)] The trajectory of the perturbed system (\ref{kpr2.5}) with the initial point $(0.7,0)$ spirals inward, converging to $(1,0)$ as $t \rightarrow \infty$  (see Figure \ref{comnbied_2}-$c$ and Figure \ref{comnbied_2c}), while the trajectory with the initial point $(0.3,0)$ spirals outward (see Figure \ref{comnbied_2}-$a$). This observation suggests the presence of an unstable limit cycle.
\item [\rm(ii)] For the perturbed system (\ref{kpr2.5}), the trajectory starting from the initial condition $(0.487,0)$ closely approximates the closed trajectory (see Figure \ref{comnbied_2}-$b$).
\end{itemize}

\begin{figure}[htbp] 
  \centering
\includegraphics[width=1.08\linewidth]{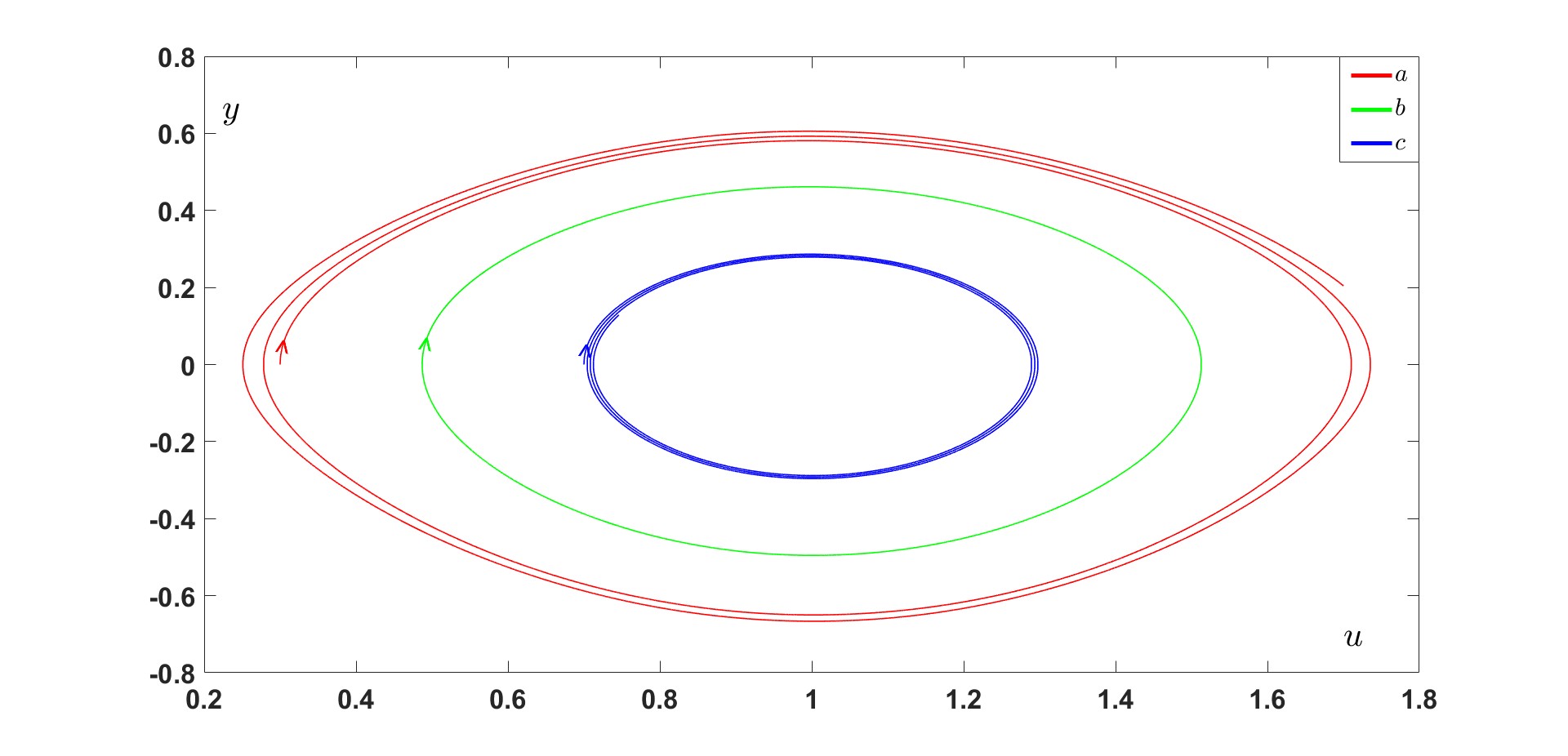}
  \caption{Trajectories of system (\ref{kpr2.5}) for $\epsilon=0.1,$ $\beta =1,$   $n=2$, $u_0=1$, $k=2$, $\alpha_0=-1.06133$, $\alpha_2=1$ with initial conditions $(a)$ $u=0.3$, $y=0,$ $(b)$ $u=0.487$, $y=0,$ $(c)$ $u=0.7$, $y=0.$} 
  \label{comnbied_2}
\end{figure}
\begin{figure}[htbp] 
  \centering
\includegraphics[width=1.08\linewidth]{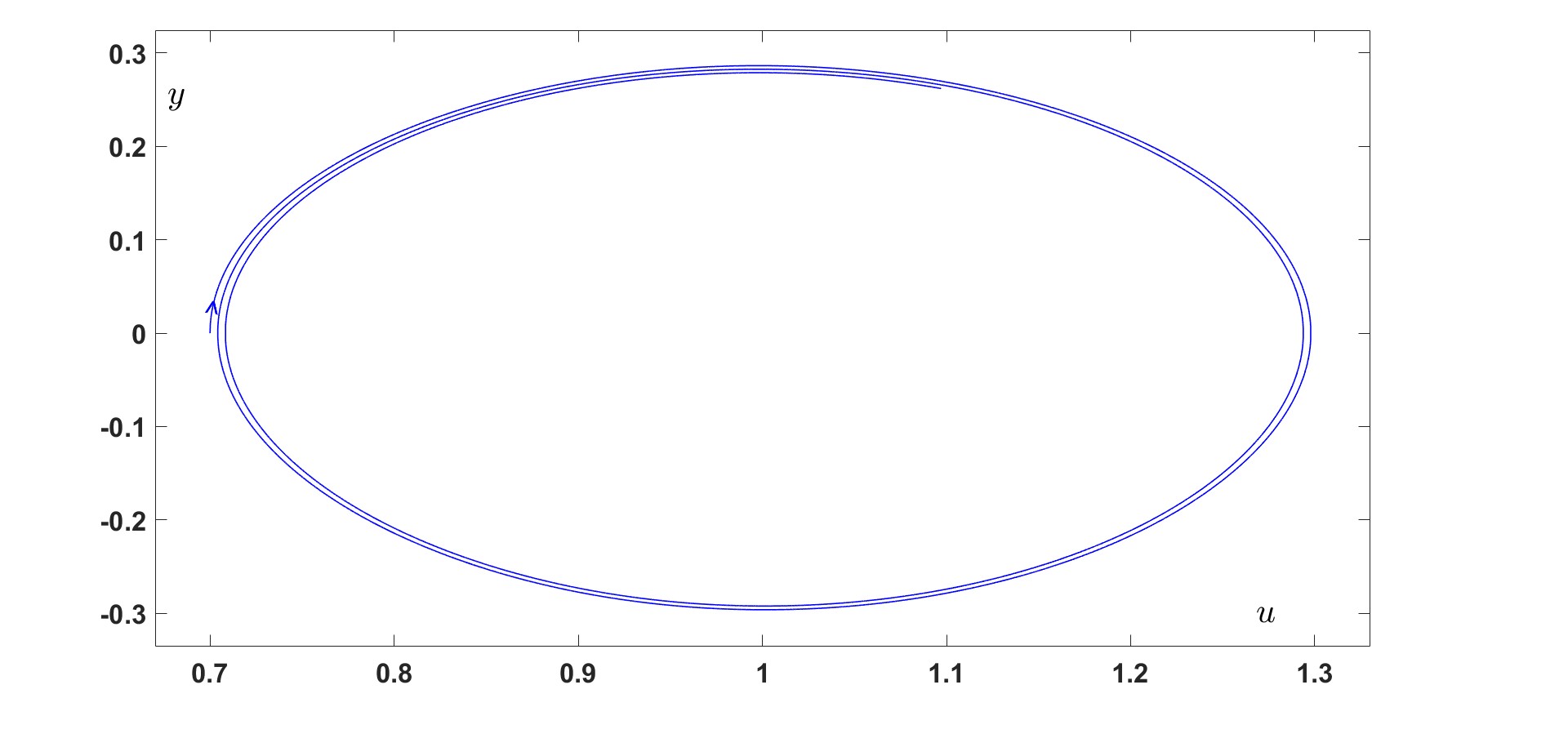}
  \caption{Zoomed version of Figure \ref{comnbied_2}(c)} 
  \label{comnbied_2c}
\end{figure}

\section{Conclusion}
In this study, we have proved existence of at most one isolated periodic wave for \rm(i) $0<u_0<k \leq 2u_0$, \rm(ii) $-k<u_0<0,$ and \rm(iii) $u_0<-k$ for all positive integer $n$. When $u_0=-k,$ for all even positive integers $n,$ there exists at most one isolated periodic wave. When $u_0=-k,$ for all odd positive integers $n,$ there exists no isolated periodic wave.

Here we used the monotonicity of the Abelian integral to prove the desired results. Also, we have obtained the interval of the ratio $\frac{\alpha_0}{\alpha_n}$ for the existence of limit cycle for a case. In the similar way, we can get the intervals for other cases also. Thus, for  specific $a_0$ and $a_n$, there exists a periodic traveling wave solution for the equation (\ref{eq3.01}).\\
\noindent\textbf{Author Contributions:} Both authors were involved in both the writing and reviewing of the manuscript.\\
\textbf{Data Availability:} No manuscript is associated with this document.
\section*{Declarations}
\textbf{Conflict of interest:} The authors state that they have no conflict of interest.


\begin{thebibliography}{99}
\bibitem{ref02}
A. A. Andronov, E. A. Leontovich, I. I. Gordon, A. G. Maier,
Qualitative Theory of Second-Order Dynamical Systems, John Wiley and Sons, New York, 1973.
\bibitem{ref30}
C. Christopher, C. Li, Limit cycles of differential equations, Birkh\"{a}user Verlag, Berlin, 2007.
\bibitem{ref24}
J. K. Hale, Ordinary Differential Equations, Wiley-Interscience, New York, 1969.
\bibitem{ref6}
C. Liu, G. Chen, Z. Sun, New criteria for the monotonicity of the ratio of two Abelian integrals, J Math Anal Appl. 465 (2018) 220-234.
\bibitem{ref9}
W. Liu, M. Han, Bifurcations of traveling wave solutions of a generalized Burgers–Fisher equation, J Math Anal Appl. 533(2) (2024) 128012.
\bibitem{ref20}
A. M. Ozorio de Almeida, Hamiltonian Systems: Chaos and Quantization, Cambridge University Press, Cambridge, UK, 1988.
\bibitem{ref8}
K. Patra, C. S. Rao, Isolated periodic traveling waves of some reaction convection diffusion equations, Stud Appl Math. 2024;e12729.
https://doi.org/10.1111/sapm.12729.
\bibitem{ref5}
K. Patra, C. S. Rao, Existence of periodic traveling wave solutions of a family of generalized Burgers-Fisher equations, (Submitted).
\bibitem{ref15}
L. Perko, Differential Equations and Dynamical Systems, vol. 7, Springer Science $\&$ Business Media, 2013.
 \bibitem{ref1}
 S. V. Petrovskii , B. L. Li, Exactly solvable models of biological invasion, Chapman \& Hall/CRC Press, Boca Raton, 2006.
 \bibitem{ref2}
S. V. Petrovskii , B. L.  Li, An exactly solvable model of population dynamics with density-dependent migrations and the Allee effect, Math Biosci. 186 (2003) 79–91.
\bibitem{ref31}
H. Poincar\'{e}, Sur le probleme des trois corps et les \'{e}quations de la dynamique,
Acta Math. XIII (1890), 1–270.
\bibitem{ref32}
L. Pontryagin, On dynamical systems close to hamiltonian ones. Zh Exp Theor Phys. 4 (1934) 234–238.
\bibitem{ref3}
X. Sun, P. Yu , B. Qin, Global existence and uniqueness of periodic waves in a population model with density-dependent migrations and Allee effect, Internat. J. Bifur. Chaos Appl. Sci. Engrg. 27(12) (2017) 1750192.
\bibitem{ref17}
Q. Wang, Y. Xiong, W. Huang, V. G. Romanovski, Isolated periodic wave trains in a generalized Burgers–Huxley equation, Electron J Qual Theory Differ Equ. 4 (2022) 1-16.
\bibitem{ref4}
Q. Wang, Y. Xiong, W. Huang, P. Yu, Isolated periodic wave solutions arising from Hopf and Poincar\'{e} bifurcations in a class of single species model, J Differ Equ. 311 (2022) 59-80.

\bibitem{ref16}
M. Wei, X. Chen, Near-Ordinary Periodic Waves of a Generalized Reaction–Convection–Diffusion Equation, Qual Theory Dyn Syst. 22(3) (2023) 107.
\bibitem{ref10}
M. Wei, X. Chen, Y. Dai, Periodic wave solutions for a generalized reaction-convection–diffusion equation of high-order, Appl Math Lett. 158 (2024) 109249.
\bibitem{ref11}
H. Zhang, Y. Xia, Periodic Wave Solution of the generalized Burgers–Fisher Equation via Abelian Integral, Qual Theory Dyn Syst. 21(3) (2022) 64.

\bibitem{ref21}
H. Zoladek, R. Murillo, Qualitative Theory of ODEs: An Introduction to Dynamical Systems Theory, World Scientific, Singapore, 2023.

\end{thebibliography}
\end{document}